\documentclass{amsart}

\usepackage{latexsym}
\usepackage{amsfonts,amssymb} 
\usepackage{graphicx}

\newcommand{\field}[1]{\mathbb{#1}}
\newcommand{\A}{\field{A}}

\newcommand{\G}{\field{G}}
\newcommand{\N}{\field{N}}
\newcommand{\Q}{\field{Q}}

\newcommand{\Z}{\field{Z}}
\newcommand{\krn}{{\rm ker}\,}

\newtheorem{theorem}{Theorem}[section]

\newtheorem{lemma}[theorem]{Lemma}
\newtheorem{corollary}[theorem]{Corollary}
\newtheorem{definition}[theorem]{Definition}
\newtheorem{remark}[theorem]{Remark}

\newtheorem{example}[theorem]{Example}

\begin{document}

\makeatletter	   
\makeatother     

\title{Cable Algebras and Rings of $\G_a$-Invariants}
\author{Gene Freudenburg and Shigeru Kuroda}
\date{\today} 
\subjclass[2010]{13A50,14R20}
\maketitle

\pagestyle{plain}

\begin{abstract} For a field $k$, the ring of invariants of an action of the unipotent $k$-group $\G_a$ on an affine $k$-variety is quasi-affine, but not generally affine. Cable algebras are introduced as a framework for studying these invariant rings. It is shown that the ring of invariants  for the $\G_a$-action on $\A^5_k$ constructed by Daigle and Freudenburg is a monogenetic cable algebra. A generating cable is constructed for this ring, and a complete set of relations is given as a prime ideal in the infinite polynomial ring over $k$. 
In addition, it is shown that the ring of invariants for the well-known $\G_a$-action on $\A^7_k$ due to Roberts is a cable algebra. 
\end{abstract}
 
\section{Introduction}

We introduce cable algebras to describe the structure of rings of invariants for algebraic actions of the unipotent group $\G_a$ on affine varieties over a ground field $k$. Winkelmann \cite{Winkelmann.03} has shown that such rings are always quasi-affine over $k$, but they are not generally affine. 
Roberts \cite{Roberts.90} gave the first example of a non-affine invariant ring for a $\G_a$-action on an affine space. 
Specifically, Roberts' example involved an action of $\G_a$ on the affine space $\A_k^7$, where $k$ is of characteristic zero. Subsequent examples of $\G_a$-actions of non-finite type were constructed by Freudenburg, and by Daigle and Freudenburg, for $\A_k^6$ and $\A_k^5$, respectively \cite{Freudenburg.00, Daigle.Freudenburg.99}. These examples are counterexamples to Hilbert's Fourteenth Problem. 

Kuroda \cite{Kuroda.04b} used 
SAGBI basis techniques to show that an infinite system of invariants constructed by Roberts for the action on $\A_k^7$ generates the invariant ring as a $k$-algebra. Tanimoto \cite{Tanimoto.06} used the same techniques to identify generating sets for the actions on $\A_k^6$ and $\A_k^5$. Our results show that Tanimoto's generating sets are not minimal; see {\it {\S}\ref{Tanimoto}}. 
From the point of view of classical invariant theory, a structural description of a ring of invariants involves determination of a minimal set of generators of the ring as a $k$-algebra, together with a minimal set of generators for the ideal of their relations. However, for an infinite set of generators, or even a large finite set of generators, such a description can be complicated, and the choice of generating set can seem arbitrary. 

When $k$ is of characteristic zero,
$\G_a$-actions on an affine $k$-variety $X$ are equivalent to locally nilpotent derivations of the coordinate ring $k[X]$, and the invariant ring $k[X]^{\G_a}$ equals the kernel of the derivation. In many cases, $k[X]^{\G_a}$ admits a non-zero locally nilpotent derivation, and this gives additional structure to exploit. 

For a commutative $k$-domain $B$, a locally nilpotent derivation $D$ of $B$ induces a directed tree structure on $B$. A {\it $D$-cable} is any complete linear subtree rooted in the kernel of $D$. The condition for $B$ to be a cable algebra is a finiteness condition: $B$ is a {\it cable algebra} if (for some $D$) $D\ne 0$ and $B$ is generated by a finite number of $D$-cables over the kernel of $D$. $B$ is a {\it simple cable algebra} if it is generated by one $D$-cable over $k$. 
Elements in the ideal of relations in the infinite polynomial ring for the generating cables are {\it cable relations}. 

To illustrate, consider a nilpotent linear operator $N$ on a finite-dimensional $k$-vector space $V$. Choose a basis $\{ x_{i,j}\, |\, 1\le i\le m,1\le j\le n_i\}$ of $V$ so that the effect of $N$ for fixed $i$ is:
\[
 x_{i,n_i}\to x_{i,n_i-1}\to\cdots\to x_{i,2}\to x_{i,1}\to 0
 \]
This defines the Jordan form of $N$, which in turn gives a cable structure on the symmetric algebra $S(V)$. In particular, $N$ induces a locally nilpotent derivation $D$ on $S(V)$, and each sequence 
$x_{i,j}$ for fixed $i$ is a $D$-cable $\hat{x}_i$, where $S(V)=k[\hat{x}_1,...,\hat{x}_m]$. In this sense, the cable algebra structure induced by a locally nilpotent derivation can be viewed as a generalization of Jordan block form for a nilpotent linear operator. 

For rings of non-finite type over $k$, the ring $S=k[x,xv,xv^2,...]$ is a prototype, where $k[x,v]$ is the polynomial ring in two variables over $k$. 
The partial derivative $\partial/\partial v$ restricts to a locally nilpotent derivation $D$ of $S$, and the infinite sequence $\frac{1}{n!}xv^n$ defines a $D$-cable 
$\hat{s}$ for which $S=k[\hat{s}]$. So $S$ is a simple cable algebra. Although $S$ is not quasi-affine, it plays an important role in our investigation. For example, one of our main objects of interest is
the ring $A$ of invariants for the $\G_a$-action on $\A^5$ constructed by Daigle and Freudenburg, and we show that $A$ admits a mapping onto $S$. 

\subsection{Description of Main Results}
We assume throughout that $k$ is a field of characteristic zero. On the polynomial ring $B=k[a,v,x,y,z]=k^{[5]}$, define the locally nilpotent derivation $D$ of $B$ by:
\[
\small
D=a^3\frac{\partial}{\partial x} + x\frac{\partial}{\partial y} +  y\frac{\partial}{\partial z}  +a^2\frac{\partial}{\partial v}
\]
For the corresponding $\G_a$-action on $X=\A^5_k$, the ring of invariants $k[X]^{\G_a}$ is not finitely generated over $k$; see \cite{Daigle.Freudenburg.99}. 

If $A=\krn D$, the kernel of $D$, then
the partial derivative $\frac{\partial}{\partial v}$ restricts to $A$, and $\partial$ denotes the restriction of $\frac{\partial}{\partial v}$ to $A$. 
We give a complete description of the ring $A$ as a cable algebra relative to $\partial$, including its relations as a cable ideal in the infinite polynomial ring $\Omega =k[x_0,x_1,x_2,...]$. Moreover, we construct a specific 
$\partial$-cable $\hat{\sigma}=(\sigma_n)$ from these relations, wherein $\sigma_{n+1}$ is expressed as an explicit rational function in $\sigma_0,...,\sigma_n$. Our proofs do not use SAGBI bases, relying instead on properties of the down operator $\Delta$ on $\Omega$, a $k$-derivation defined by:
\[
\Delta x_i=x_{i-1}\,\, (i\ge 1)\quad {\rm and}\quad \Delta x_0=0
\]
Let $\Omega [t]=\Omega^{[1]}$ and extend $\Delta$ to $\tilde{\Delta}$ on $\Omega [t]$ by $\tilde{\Delta}t=0$. 
\medskip

\begin{itemize}
\item [] {\bf Generators} (Theorem~\ref{A-generators}):
{\it There exists an infinite homogeneous $\partial$-cable $\hat{s}$ rooted at $a$, and for any such $\partial$-cable we have $A=k[h,\hat{s}]$ for $h\in\krn\partial$. Moreover, this is a minimal generating set for $A$ over $k$. }
\smallskip
\item [] {\bf Relations} (Theorem~\ref{A-relations}): {\it There exists an ideal $\mathcal{I}=(\hat{\Theta}_4, \hat{\Theta}_6, \hat{\Theta}_8,...)$ in $\Omega [t]$ generated by quadratic homogeneous $\tilde{\Delta}$-cables 
$\hat{\Theta}_n$ such that:
\[
A\cong \Omega [t]/\mathcal{I} 
\]
The cables $\hat{\Theta}_n$ depend on the choice of $\partial$-cable $\hat{s}$ and can be constructed explicitly from $\hat{s}$.}
\medskip
\item [] {\bf Constructs} (Theorem~\ref{A-constructs}): {\it Let $A_a$ be the localization of $A$ at $a$, and define a sequence $\sigma_n\in A_a$ by
$\sigma_0=a$ and:
\[\textstyle
\sigma_1=av-x\,\, ,\,\, \sigma_2 = \frac{1}{2}(av^2-2xv+2a^2y) \,\, ,\,\, 
\sigma_3 = \frac{1}{6}(av^3-3xv^2+6a^2yv-6a^4z)
\]
Given $n\ge 4$, let $e\ge 1$ be such that $-2\le n-6e\le 3$. If $\sigma_0,...,\sigma_{n-1}\in A_a$ are known, define $\sigma_n\in A_a$ implicitly as follows.
\medskip
\begin{itemize}
\item [(i)] If $n=6e-2$ or $n=6e+2$, then $\sum_{i=0}^n(-1)^i\sigma_i\sigma_{n-i} = 0$
\medskip
\item [(ii)] If $n=6e-1$ or $n=6e+3$, then $\sum_{i=1}^n(-1)^ii\sigma_i\sigma_{n-i} = 0$
\medskip
\item [(iii)] If $n=6e$, then $\sum_{i=0}^{n+2}(-1)^i\left( 3i(i-1)-n(n+2)\right) \sigma_i\sigma_{n+2-i} = 0$
\medskip
\item [(iv)] If $n=6e+1$, then $\sum_{i=1}^{n+3}(-1)^{i+1}\left( (i-1)(i-2)-n(n+2)\right) i \sigma_i\sigma_{n+3-i} = 0$
\end{itemize}
\medskip
Then $\sigma_n\in A$ for each $n\ge 0$ and $\hat{\sigma}=(\sigma_n)$ is a $\partial$-cable rooted at $a$. }
\end{itemize}
\medskip

As seen in these results, quadratic relations in $\Omega$ are especially important. 
A basis for the vector space of quadratic forms in $\krn\Delta$ is given by $\{ \theta_n^{(0)}\, |\, n\in 2\N \}$, where:
\[
\theta_n^{(0)} = \sum_{i=0}^n(-1)^ix_ix_{n-i}
\]
If $\{ \hat{\theta}_n\}$ is any system of quadratic $\Delta$-cables with $\hat{\theta}_n$ rooted at $\theta_n^{(0)}$, then the vertices of these cables form a basis for
$\Omega_2$, the space of quadratic forms in $\Omega$ ({\it Lem.~\ref{lemma1}}). Moreover, the quadratic ideals 
\[
\mathcal{Q}_n=(\hat{\theta}_n,\hat{\theta}_{n+2},\hat{\theta}_{n+4},...) \,\, ,\,\, n\in 2\N 
\]
are independent of the system of cables chosen ({\it Thm.~\ref{Q-ideal}}). These ideals, called {\it fundamental Q-ideals}, are intrinsically important to the theory at hand. Compare this to the linear case: The only linear form in $\krn\Delta$ is $x_0$, up to constant, and if $\hat{L}=(L_n)$ is any homogeneous $\Delta$-cable rooted at $x_0$, then the linear forms $L_n$, $n\ge 0$, form a basis of 
the space of linear forms $\Omega_1$, and we have equality of $\Omega$-ideals:
\[
(\hat{L})=(L_0,L_1,L_2,...)=(x_0,x_1,x_2,...)
\]
Therefore, $\Omega /(\hat{L})=k$ and $(\hat{L})$ is a maximal ideal of $\Omega$. 

We show the following.
\begin{itemize}
\medskip
\item [] (Theorem~\ref{S-relations}) {\it $\mathcal{Q}_2$ is a prime ideal of $\Omega$ and $\Omega/\mathcal{Q}_2\cong_kS$, 
where $S\subset k[x,v]=k^{[2]}$ is the simple cable algebra of non-finite type and of transcendence degree 2 over $k$ defined by $S=k[x,xv,xv^2,...]$.}
\medskip
\item [] (Theorem~\ref{Abar-relations}) {\it $\mathcal{Q}_4$ is a prime ideal of $\Omega$ and $\Omega/\mathcal{Q}_4\cong_kA/hA$, which is a simple cable algebra of non-finite type and of transcendence degree 3 over $k$. }
\end{itemize}
\medskip

Finally, we show that the ring of invariants for the Roberts action in dimension 7 is a cable algebra. On the polynomial ring
$k[X,Y,Z,S,T,U,V]$, define the locally nilpotent derivation:
\[\small
\mathcal{D}_2=X^3\frac{\partial}{\partial S} +Y^3 \frac{\partial}{\partial T} +Z^3 \frac{\partial}{\partial U}  + (XYZ)^2\frac{\partial}{\partial V}
\]
$\mathcal{D}_2$ commutes with the 3-cycle $\alpha$ defined by $\alpha (X,Y,Z,S,T,U,V)=(Z,X,Y,U,S,T,V)$. 
The partial derivative $\partial /\partial V$ restricts to the kernel $\mathcal{A}_2$ of $\mathcal{D}_2$, and $\delta_2$ denotes the restricted derivation. 
\begin{itemize}
\medskip
\item [] (Theorem~\ref{Roberts}) {\it There exists a $\delta_2$-cable $\hat{P}$ in $\mathcal{A}_2$ rooted at $X$, and for any such $\delta_2$-cable,
\[
\mathcal{A}_2 = k[H_2,\alpha H_2,\alpha^2 H_2, \hat{P},\alpha\hat{P},\alpha^2\hat{P}]
\]
where $H_2\in\krn\delta_2$. }
\end{itemize}

\subsection{Additional Background}

Let $K$ be any field. For $n\le 3$, the ring of invariants of a $\G_a$-action on $\A^n_K$ is of finite type, due to a fundamental theorem of Zariski. 
It is not known if the ring of invariants of a $\G_a$-action on $\A^4_K$ is always of finite type; see {\it \S\ref{dim4}}. 
According to the classical Mauer-Weitzenb\"ock Theorem, if the characteristic of $K$ is zero, then $K[\A^n_K]^{\G_a}$ is of finite type when 
$\G_a$ acts on $\A^n_K$ by linear transformations. However, it is not known if this is true for all fields. 
To date, there is no known example of a field $K$ of positive characteristic and a $\G_a$-action on $\A^n_K$ for which 
$K[\A^n_K]^{\G_a}$ is of non-finite type.

\subsection{Outline}

\begin{itemize}  
\item [1.] Introduction
\item [2.] Locally Nilpotent Derivations
\item [3.] Cables and Cable Algebras
\item [4.] The Derivation $D$ in Dimension Five
\item [5.] Generators of $\bar{A}$ and $A$
\item [6.] Relations in $\bar{A}$
\item [7.] Relations in $A$
\item [8.] Roberts' Derivations in Dimension Seven
\item [9.] Further Comments and Questions
\end{itemize} 


\section{Locally Nilpotent Derivations}\label{LNDs}  
Let $k$ be a field of characteristic zero and let $B$ be a commutative $k$-domain. A {\it locally nilpotent derivation} of $B$ is a derivation $D:B\to B$ such that, for each $b\in B$, there exists $n\in\N$ (depending on $b$) such that $D^nb=0$. 
Let $\krn D$ denote the kernel of $D$. The set of locally nilpotent derivations of $B$ is denoted by ${\rm LND}(B)$. Note that $k\subset \ker D$ for any $D\in\text{LND}(B)$
(c.f.~{\it Principle 1} of \cite{Freudenburg.06}). 

It is well known that the study of $\G_a$-actions on an affine $k$-variety $X$ is equivalent to the study of locally nilpotent derivations on the corresponding coordinate ring $k[X]$. In particular, the action induced by $D\in {\rm LND}(B)$ is given by the exponential map $\exp (tD)$, $t\in\G_a$, and $k[X]^{\G_a}=\krn D$. 

In this section, we give some of the basic properties for rings with locally nilpotent derivations. The reader is referred to \cite{Freudenburg.06} for further details of the subject. 

\subsection{Basic Definitions and Properties} Given $D\in {\rm LND}(B)$, 
if $A=\krn D$, then $A$ is filtered by the {\it image ideals} :
\[
I_n:=A\cap D^nB \quad (n\ge 0) \quad {\rm and}\quad I_{\infty}:=\cap_{n\ge 0}I_n
\]
Note that $I_0=A$ and $I_{n+1}\subset I_n$ for $n\ge 0$. $I_1$ is called the {\it plinth ideal} for $D$, and $I_{\infty}$ is called the {\it core ideal} for $D$. 

A {\it slice} for $D$ is any $s\in B$ such that $Ds=1$. Note that $D$ has a slice if and only if $D:B\to B$ is surjective.  

A {\it local slice} for $D$ is any $s\in B$ such that $D^2s=0$ but $Ds\ne 0$. For a local slice $s\in B$ of $D$, let $B_{Ds}$ and $A_{Ds}$ denote the localizations of $B$ and $A$ at $Ds$, respectively. Then $B_{Ds}=A_{Ds}[s]$, where $s$ is transcendental over $A_{Ds}$. 
Given $b\in B$, $\deg_Db$ is the degree of $b$ as a polynomial in $s$, which is independent of the choice of local slice $s$.
The corresponding {\it Dixmier map} $\pi_s:B_{Ds}\to A_{Ds}$ is the algebra map defined by:
\[
\pi_s(f) = \sum_{i\ge 0}\frac{(-1)^i}{i!}D^if\cdot \left(\frac{s}{Ds}\right)^i
\]
If $E$ is any $k$-derivation of $B$ which commutes with $D$, then it is immediate from this definition that:
\begin{equation}\label{commute}
E\pi_s(f) = \pi_s(Ef)-\pi_s(Df)E(s/Ds)
\end{equation}

Let $S\subset B$ be a non-empty subset, and let $k\subset R\subset A$ be a subring.  Define the subring:
\[
R[S,D]=R[D^is\, |\, s\in S, i\ge 0]
\]
Note that $D$ restricts to $R[S,D]$, and that $R[S,D]$ is the smallest subring of $B$ containing $R$ and $S$ to which $D$ restricts. 

\subsection{The Down Operator} Let $\Omega =k[x_0,x_1,x_2,...]$ be the infinite polynomial ring, and let $\Omega_+$ be the ideal of $\Omega$ defined by:
\[
\Omega_+= \sum_{n\ge 0}x_n\cdot\Omega
 \]
Let $\Delta\in {\rm LND}(\Omega )$ denote the {\it down operator} on $\Omega$:
\[
\Delta x_n=x_{n-1} \,\, (n\ge 1) \quad {\rm and}\quad \Delta x_0=0
\]
Then $\Delta :\Omega_+\to\Omega_+$ is surjective  (\cite{Freudenburg.13},Thm.{\,}3.1). 

$\Omega$ has a $\Z^2$-grading defined by $\deg x_i=(1,i)$, where each $x_i$ is homogeneous ($i\ge 0$). For this grading, $\Delta$ is homogeneous and $\deg\Delta =(0,-1)$. 
Given $r,s\ge 0$, let $\Omega_{(r,s)}$ denote the set of elements of $\Omega$ of degree $(r,s)$, and let $\Omega_r=\sum_s\Omega_{(r,s)}$.  
Then $\Delta :\Omega_{(r,s)}\to\Omega_{(r,s-1)}$ is surjective for each $r,s\ge 1$. 

\subsection{Tree Structure Induced by an LND}
Let $B$ be a commutative $k$-domain. To any $D\in {\rm LND}(B)$ we associate the rooted tree ${\rm Tr}(B,D)$ whose vertex set is $B$, and whose (directed) edge set consists of pairs $(f,Df)$, where $f\ne 0$. 
Equivalently, ${\rm Tr}(B,D)$ is the tree defined by the partial order on $B$ defined by: $a\le b$ iff $D^nb=a$ for some $n\ge 0$. 

Let $A=\krn D$. 

\begin{itemize}
\item [(i)]  Given $a,b\in B$ with $b\ne 0$, $b$ is a {\it predecessor} of $a$ if and only if $a$ is a {\it successor} of $b$ if and only if $a<b$.
Similarly, $b$ is an {\it immediate predecessor} of $a$ if and only if $a$ is an {\it immediate successor} of $b$ if and only if $Db=a$. 
\smallskip
\item [(ii)] The {\it terminal vertices} of ${\rm Tr}(B,D)$ are those without predecessors, i.e., elements of $B\setminus DB$. 
If $D$ has a slice, i.e., $DB=B$, then ${\rm Tr}(B,D)$ has no terminal vertices. 
\smallskip
\item [(iii)] Every subtree $X$ of ${\rm Tr}(B,D)$ has a unique root, denoted ${\rm rt}(X)$. 
\smallskip
\item [(iv)] A subtree $X$ of ${\rm Tr}(B,D)$ is {\it complete} if every vertex of $X$ which is not terminal in ${\rm Tr}(B,D)$ has at least one predecessor in $X$. 
\smallskip 
\item [(v)] A subtree $X$ of ${\rm Tr}(B,D)$ is {\it linear} if every vertex of $X$ has at most one immediate predecessor in $X$. 
\smallskip
\item [(vi)] If $B$ is graded by an abelian group, then any homogeneous $b\in B$ is a {\it homogeneous vertex} of ${\rm Tr}(B,D)$.  
A subtree $X$ of ${\rm Tr}(B,D)$ is {\it homogeneous} if every $b\in {\rm vert}(X)$ is homogeneous. 
If $D$ is homogeneous, then the {\it full homogeneous subtree} is the subtree of ${\rm Tr}(B,D)$ spanned by the homogeneous vertices.  
\end{itemize}

\section{Cables and Cable Algebras}

\subsection{$D$-Cables} \label{D-cables}

\begin{definition}{\rm Let $B$ be a commutative $k$-domain and let $D\in {\rm LND}(B)$. 
A {\it $D$-cable} is a complete linear subtree $\hat{s}$ of ${\rm Tr}(B,D)$ rooted at a non-zero element of $\krn D$. 
$\hat{s}$ is a {\it terminal} $D$-cable if it contains a terminal vertex. $\hat{s}$ is an {\it infinite} $D$-cable if it is not terminal.}
\end{definition}
We make several remarks and further definitions, assuming that $B$ is a commutative $k$-domain, $D\in {\rm LND}(B)$, $I_n=\krn D\cap D^nB$ ($n\ge 0$) and $I_{\infty}=\cap_{n\ge 0}I_n$. 
\begin{itemize}
\item [(i)]  If $\hat{s}$ is a $D$-cable, then $\hat{s}$ is terminal if and only if its vertex set is finite, and $\hat{s}$ is infinite if and only if $\hat{s}\subset DB$. 
\smallskip
\item [(ii)] A $D$-cable is denoted by $\hat{s}=(s_j)$, where $s_j\in B$ for $j\ge 0$ and $Ds_j=s_{j-1}$ for $j\ge 1$. 
It is rooted at $s_0\in\krn D$, which is non-zero. For multiple $D$-cables $\hat{s}_1,...,\hat{s}_n$, we will write $\hat{s}_i=(s_i^{(j)})$ for $1\le i\le n$ and $j\ge 0$. 
\smallskip
\item [(iii)] The {\it length} of a $D$-cable $\hat{s}$ is the number of its edges (possibly infinite), denoted ${\rm length}(\hat{s})$. 
If $\hat{s}=(s_n)$ and $N={\rm length}(\hat{s})$, then $s_0\in I_N$, and if $\hat{s}$ is terminal, then $s_N$ is its terminal vertex. 
\smallskip
\item [(iv)] Every $b\in\krn D\setminus DB$ is a terminal vertex of ${\rm Tr}(B,D)$ and defines a terminal $D$-cable of length zero. 
\smallskip
\item [(v)] If $B$ is graded by an abelian group, then a $D$-cable is {\it homogeneous} if it is a homogeneous subtree of ${\rm Tr}(B,D)$. 
\smallskip
\item [(vi)] Every non-zero vertex $b\in B$ belongs to a $D$-cable. If two $D$-cables $\hat{s}=(s_n)$ and $\hat{t}=(t_n)$ have $s_m=t_n$ for some $m,n\ge 0$, then $m=n$ and 
$s_i=t_i$ for all $i\le m$. If $\hat{s}$ and $\hat{t}$ share an infinite number of vertices, then $\hat{s}=\hat{t}$.
\smallskip
\item [(vii)] Suppose that $B^{\prime}\subset B$ is a subset with $DB^{\prime}\subset B^{\prime}$. If $\hat{s}\subset B$ is a $D$-cable such that either $\hat{s}\cap B^{\prime}$ is infinite, or $\hat{s}$ is terminal of length $N$ and $s_N\in B^{\prime}$, then $\hat{s}\subset B^{\prime}$.  
\smallskip
\item [(viii)] If $P\in\Omega$ is a polynomial in $x_0,...,x_n$ and $\hat{s}$ is a $D$-cable of length at least $n$, then $P(\hat{s})$ means $P(s_0,...,s_n)$.
\smallskip
\item [(ix)] Given $D$-cables $\hat{s}_1,...,\hat{s}_n$ for $n\ge 0$, the notation $k[\hat{s}_1,...,\hat{s}_n]$ (respectively $(\hat{s}_1,...,\hat{s}_n)$) indicates the $k$-subalgebra of $B$ (respectively, ideal of $B$) generated by the vertices of $\hat{s}_i$ for $1\le i\le n$.
\smallskip
\item [(x)] Let $\hat{s}=(s_n)$ be a $D$-cable of length $N$. If $\hat{s}$ is terminal, define the map $\phi_{\hat{s}} :k^{[N+1]}\to B$ by $\phi _{\hat{s}}(x_i)=s_i$ 
for $0\le i\le N$. If $\hat{s}$ is infinite, define $\phi_{\hat{s}} :\Omega\to B$ by $\phi_{\hat{s}} (x_i)=s_i$ for all $i\ge 0$. Elements of $\krn\phi_{\hat{s}}$ are the {\it cable relations} associated to $\hat{s}$. Note that 
$D\phi_{\hat{s}} = \phi_{\hat{s}}\Delta$ where $\Delta$ is the down operator on $\Omega$ or its restriction to $k^{[N+1]}$. 
\smallskip
\item[(xi)] Extend $D$ to a derivation $D^*$ on $B[t]=B^{[1]}$ by $D^*t=0$. If $\hat{s}(t)=(s_n(t))$ is a $D^*$-cable 
and $\alpha \in \ker D$ is such that 
$s_0(\alpha )\ne 0$, then $\hat{s}(\alpha )=(s_n(\alpha ))$ is a $D$-cable rooted at $s_0(\alpha )$.
\end{itemize} 
\medskip

\begin{example} {\rm 
Let $\Omega =k[x_0,x_1,x_2,...]$ be the infinite polynomial ring and $\Delta\in {\rm LND}(\Omega )$ the down operator. Then $\hat{x}=(x_j)_{j\ge 0}$ is an infinite $\Delta$-cable, 
$x_0\in I_{\infty}$, and $\Omega = k[\hat{x}]$. Relabel the variables $x_i$ by $y_n^{(j)}$ so that $\Omega = k[x_0, y_n^{(j)}\, |\, n\ge 1, 1\le j\le n]$. Define $\tilde{\Delta}\in {\rm LND}(\Omega )$ so that, for $n\ge 1$:
\[
\tilde{\Delta}: y_n^{(n)}\to y_n^{(n-1)}\to\cdots\to y_n^{(1)}\to y_n^{(0)}:=x_0\to 0
\]
Then $\hat{y}_n:=(y_n^{(j)})_{0\le j\le n}$ is a terminal $\tilde{\Delta}$-cable rooted at $x_0$ of length $n$ for each $n\ge 1$. If $\tilde{I}_{\infty}$ is the core ideal for $\tilde{\Delta}$, then $x_0\in\tilde{I}_{\infty}$ but there is no infinite $\tilde{\Delta}$-cable rooted at $x_0$, since otherwise there would exist a homogeneous 
infinite $\tilde{\Delta}$-cable rooted at $x_0$. It is easy to check that this is not the case.
}
\end{example}

In general, addition of $D$-cables is not well-defined, since $B\setminus DB$ is not closed under addition. However, an infinite $D$-cable $\hat{s}$ has $\hat{s}\subset DB$, and addition can be defined for certain pairs in this case. The next result gives basic operations on $D$-cables. These follow immediately from the definitions. 

\begin{lemma}\label{ops} Let $B$ be a commutative $k$-domain, let $D\in {\rm LND}(B)$, and let $A=\krn D$. 
\begin{itemize} 
\item [{\bf (a)}] If $\hat{s}=(s_n)$ is a $D$-cable and $a\in A$ is non-zero, then $a\hat{s}:=(as_n)$ is a $D$-cable of the same length. 
\smallskip
\item [{\bf (b)}] If $\hat{s}=(s_n)$ and $\hat{t}=(t_n)$ are infinite $D$-cables and $s_0+t_0\ne 0$, 
then $\hat{s}+\hat{t}:=(s_n+t_n)$ is an infinite $D$-cable. 
\smallskip
\item [{\bf (c)}] If $\hat{s}=(s_n)$ and $\hat{t}=(t_n)$ are infinite $D$-cables and $m\in\Z$ has $m\ge 1$, define the sequence $u_n\in B$ by $u_n=s_n$ if $n<m$ and $u_n=s_n+t_{n-m}$ if $n\ge m$. 
Then $\hat{u} := (u_n)$ is an infinite $D$-cable. 
\end{itemize}
\end{lemma}

\begin{definition} {\rm The $D$-cable $\hat{u}$ in part (c) of {\it Lem.~\ref{ops}} is called the {\it $m$-shifted sum} of $\hat{s}$ and $\hat{t}$, and is denoted by
$\hat{u}=\hat{s}+_m\hat{t}$.} 
\end{definition}

\begin{definition} {\rm Let $I\subset\N$ be either $\N \setminus \{ 0\} $ or $\{ 1,2,\ldots ,t\} $ for some integer $t\ge 1$.
Suppose that a sequence $\vec{s}=\{\hat{s}_i\}_{i\in I}$ of infinite $D$-cables is given, together with a strictly increasing sequence
$\vec{m}=\{ m_i\}_{i\in I}$ of positive integers, and a sequence $\vec{c}=\{ c_i\}_{i\in I}$ with $c_i\in \ker D\setminus \{ 0\}$ for all $i\in I$. 
Define a sequence of $D$-cables $\hat{u}_i$ rooted at $s_1^{(0)}$ inductively by:
\[
\hat{u}_1 = \hat{s}_1 \quad {\rm and}\quad
\hat{u}_i = \hat{u}_{i-1} +_{m_i}c_i\hat{s}_i \quad {\rm for}\,\, i\in I \, , i\ge 2
\]
Note that, if $\hat{u}_i=(u_i^{(j)})$, then given $j\ge 0$, there exist $u^{(j)}\in B$ and an integer $N_j$ such that $u_i^{(j)}=u^{(j)}$ for all $i\in I$ with $i\ge N_j$. 
The $D$-cable $\hat{u}:=(u^{(j)})$ so obtained is rooted at $s_1^{(0)}$ and is denoted by: }
\[
\hat{u}=\lim (\vec{s},\vec{m},\vec{c})
\]
\end{definition}
Note that, in this definition, we have $\hat{u}=\hat{u}_t$ in case $I=\{ 1,2,\ldots ,t\} $. 

\begin{example} {\rm Let $B$ be a commutative $k$-domain and $D\in {\rm LND}(B)$. Given non-zero $f\in\krn D$, let $\exp (fD):B\to B$ be the corresponding exponential automorphism of $B$. If $\hat{s}=(s_n)$ is a $D$-cable, then:
\[
D\exp (fD)(s_n)=\exp (fD)(s_{n-1}) \,\, {\rm for}\,\, n\ge 1
\]
Note that $\exp (fD)(s_0)=s_0$, and that $s_i\in DB$ if and only if $\exp (fD)(s_i)\in DB$. Therefore, $\exp (fD)(\hat{s}) := (\exp (fD)(s_n))$ defines a $D$-cable rooted at $s_0$. If $\hat{s}$ is infinite, then it is given by:}
\[
\textstyle\exp (fD)(\hat{s}) = \lim (\vec{s},\vec{m},\vec{c}) \,\, ,\,\, {\rm where}\,\, \vec{s}=(\hat{s},\hat{s},\hat{s},...)\,\, ,
\,\, \vec{m}=(1,2,3,...) \,\,  {\rm and}\,\,  \vec{c}=(f,\frac{1}{2!}f^2,\frac{1}{3!}f^3,...)
\]
\end{example}

\subsection{Quadratic $\Delta$-Cables}
Note that we can view the vector space $\Omega_1$ as being generated by the vertices of the $\Delta$-cable $\hat{x}=(x_n)$. Similarly, $\Omega_2$ admits a basis of homogeneous $\Delta$-cables.

\subsubsection{Monomial Basis}\label{monomial} Given $n\ge 0$, the {\it monomial basis} for $\Omega_{(2,n)}$ is:
\[
\{ x_0x_n,x_1x_{n-1},...,x_{\frac{n}{2}}^2\} \,\, (n\,\, {\rm even})\quad {\rm or}\quad \{ x_0x_n,x_1x_{n-1},...,x_{\frac{n-1}{2}}x_{\frac{n+1}{2}}\} \,\, (n\,\, {\rm odd})
\]
Therefore, $\dim\Omega_{(2,n)}$ equals $(n+2)/2$ if $n$ is even, or $(n+1)/2$ if $n$ is odd. 

\subsubsection{$\Delta$-Basis}\label{Delta-Basis} Given $n\in 2\N$, define $\theta_n^{(0)}\in\Omega_{(2,n)}\cap\krn\Delta$ by:
\[
\theta_n^{(0)} = \sum_{0\le i\le n}(-1)^ix_ix_{n-i} 
\]
Note that since $n$ is even, $\theta_n^{(0)}\ne 0$. Since $\Delta:\Omega_{(2,s+1)}\to\Omega_{(2,s)}$ is surjective for all $s\ge 0$, there exists a homogeneous $\Delta$-cable $\hat{\theta}_n=(\theta_n^{(j)})$ rooted at $\theta_n^{(0)}$.
By definition, we have $\theta_n^{(j)}\in\Omega_{(2,n+j)}$
for each $j\ge 0$. By {\it $\S$\,\ref{monomial}}, $\krn\Delta\cap\Omega_{(2,s)}$ equals $\{0\}$ if $s$ is odd, and equals $k\cdot\theta_s^{(0)}$ if $s$ is even
(c.f. Cor.3.3 of \cite{Freudenburg.13}).
Therefore,  $\Delta :\Omega_{(2,n+1)}\to\Omega_{(2,n)}$ is an isomorphism.  
It follows that, if $\hat{\theta}_n=(\theta_n^{(j)})$ is any homogeneous $\Delta$-cable rooted at $\theta_n^{(0)}$, then $\theta_n^{(1)}$ is uniquely determined. It is given by:
\[
\theta_n^{(1)}=\sum_{i=1}^{n+1}(-1)^{i+1}ix_ix_{n+1-i}
\]
\begin{lemma}\label{lemma1} Let $\{\hat{\theta}_n \, |\, n\in 2\N\}$ be a set of homogeneous $\Delta$-cables such that $\hat{\theta}_n=(\theta_n^{(j)})$ is rooted at $\theta_n^{(0)}$.
\begin{itemize}
\item [{\bf (a)}] Given $j\ge 0$, 
the set $\{ \theta_{2i}^{(j-2i)}\mid 0\le i\le j/2\} $ is a basis for $\Omega_{(2,j)}$. 
\smallskip
\item [{\bf (b)}] The vertices of $\hat{\theta}_n$ $(n\in 2\N )$ form a basis for $\Omega_2$. 
\end{itemize}
\end{lemma}

\begin{proof}  To prove part (a), we proceed by induction on $j\ge 0$. We have that:
\[
\Omega_{(2,0)}=\langle x_0^2\rangle = \langle \theta_0^{(0)} \rangle
\]
So the statement of part (a) holds if $j=0$. 

Assume that, for $j\ge 1$, the set 
$\{ \theta_{2i}^{(j-1-2i)}\mid 0\le i\le (j-1)/2\}$
forms a basis for $\Omega_{(2,j-1)}$. If $j$ is odd, then $\Delta :\Omega_{(2,j)}\to\Omega_{(2,j-1)}$ is an isomorphism, and the set
$\{ \theta_{2i}^{(j-2i)}\mid 0\le i\le j/2\}$
is a basis for $\Omega_{(2,j)}$. If $j$ is even, then the kernel of $\Delta :\Omega_{(2,j)}\to\Omega_{(2,j-1)}$ is $k\cdot\theta_j^{(0)}$, and again we conclude that $\{ \theta_{2i}^{(j-2i)}\mid 0\le i\le j/2\}$ is a basis for $\Omega_{(2,j)}$. This proves part (a).

Part (b) is an immediate consequence of part (a).
\end{proof}

\begin{definition} {\rm Any set of $\Delta$-cables
$\{\hat{\theta}_n\}$ of the type described in {\it Lem.~\ref{lemma1}(b)} is a {\it $\Delta$-basis} for $\Omega_2$.}
\end{definition}

\subsubsection{Balanced $\Delta$-Basis} We define a particular $\Delta$-basis for $\Omega_2$ using binomial coefficients $\binom{i}{j}$. 
Given $n\in 2\N$ and $j\in\N$, define $\beta_n^{(j)}\in\Omega_{(2,n+j)}$ by:
\[
\beta_n^{(j)}=\sum_{i=j}^{n+j}(-1)^{j+i}\binom{i}{j}x_ix_{n+j-i}
\]
Note that $\beta_n^{(0)}=\theta_n^{(0)}$. 
\begin{lemma} If $n\in 2\N$ and $j\ge 1$, then $\Delta\beta_n^{(j)}=\beta_n^{(j-1)}$.
\end{lemma}

\begin{proof}  If $n\ge 1$ and $c_0,...,c_n\in k$, then:
\begin{equation}\label{derive}
\Delta\sum_{i=0}^nc_ix_ix_{n-i}=\sum_{i=0}^{n-1}(c_{i+1}+c_i)x_ix_{n-1-i}
\end{equation}
Given $i\in\N$ with $0\le i\le j$, we extend the definition of binomial coefficient by setting $\binom{i}{j}=0$. 
Then for all $i,j\in\N$ we have:
\[
\binom{i}{j}+\binom{i}{j-1}=\binom{i+1}{j}
\]
In addition, we can write:
\[
\beta_n^{(j)}=\sum_{i=0}^{n+j}(-1)^{j+i}\binom{i}{j}x_ix_{n+j-i}
\]
By equation (\ref{derive}) we have:
\begin{eqnarray*}
\Delta\beta_n^{(j)} &=& \sum_{i=0}^{n+j-1}\left( (-1)^{j+i+1}\binom{i+1}{j} + (-1)^{j+i}\binom{i}{j}\right) x_ix_{n+j-1-i} \\
&=& \sum_{i=0}^{n+j-1}(-1)^{j+i+1}\left( \binom{i+1}{j} - \binom{i}{j}\right) x_ix_{n+j-1-i} \\
&=& \sum_{i=0}^{n+j-1}(-1)^{j-1+i}\binom{i}{j-1} x_ix_{n+j-1-i} \\
&=& \beta_n^{(j-1)}
\end{eqnarray*}
\end{proof} 

We thus see that $\hat{\beta}_n = (\beta_n^{(j)})$ defines a homogeneous $\Delta$-cable rooted at $\theta_n^{(0)}$, and that $\{\hat{\beta}_n\}$ is a $\Delta$-basis for $\Omega_2$, which we call the {\it balanced $\Delta$-basis}.

Note that each $\beta_n^{(j)}$ involves at most $n+1$ monomials. Moreover, the monomials $x_ix_{n+j-i}$ ($j\le i\le n+j$) are linearly independent if $j\ge n$, meaning that
$\beta_n^{(j)}$ involves exactly $n+1$ monomials when $j\ge n$.

\subsubsection{Small $\Delta$-Basis} Given $n\in 2\N$ and $j\in\N$, let
\[
W_n^{(j)}=\langle x_0x_{n+j},x_1x_{n+j-1},...,x_{\frac{n}{2}}x_{\frac{n}{2}+j}\rangle
\]
noting that $W_n^{(j)}\subset\Omega_{(2,n+j)}$ 
and $\dim W_n^{(j)}=n/2+1$ for all $j\ge 0$. 
Then $\Delta :W_n^{(j+1)}\to W_n^{(j)}$ is an isomorphism, 
since $\theta _{n+j+1}^{(0)}\not\in W_n^{(j+1)}$ if $j$ is odd. 
Since $\theta_n^{(0)}\in W_n^{(0)}$, we conclude that there exists a unique $\Delta$-cable $\hat{\eta}_n=(\eta_n^{(j)})$ rooted at $\theta_n^{(0)}$ such that $\eta_n^{(j)}\in W_n^{(j)}$ for each $j\ge 0$. 
We call $\{ \hat{\eta}_n\}$ the {\it small $\Delta$-basis} for $\Omega_2$. Note that each $\eta_n^{(j)}$ involves at most $\frac{n}{2}+1$ monomials.

It is easy to check that the first three cables in this basis are given by:
\[
\textstyle
\eta_0^{(j)}=x_0x_j \,\, ,\,\, \eta_2^{(j)}=(j+2)x_0x_{2+j}-x_1x_{1+j}\,\, , \,\, \eta_4^{(j)}=\frac{(j+1)(j+4)}{2}x_0x_{4+j}-(j+2)x_1x_{3+j}+x_2x_{2+j}
\]
In particular, $\hat{\eta}_4$ will be used to give certain 3-term recursion relations; see {\it Remark~\ref{3-term}}.

\subsubsection{ $Q$-Ideals}

\begin{definition}\label{def:Q-ideal} {\rm Let $\{\hat{\theta}_n\}$ be a $\Delta$-basis for $\Omega_2$. }
\begin{itemize}
\item [(1)] {\rm A {\it $Q$-ideal} is an ideal of $\Omega$ generated by $\{ \hat{\theta}_n\, |\, n\in S\}$, where $S\subset 2\N$ is any non-empty subset.}
\item [(2)] {\rm Given $n\in 2\N$, the corresponding {\it fundamental $Q$-ideal} is: }
\[
\mathcal{Q}_n=(\hat{\theta}_n,\hat{\theta}_{n+2},\hat{\theta}_{n+4},...)
\]
\end{itemize}
\end{definition}
Note that $\Delta\hat{\theta}_n=\hat{\theta}_n$ for each $n\in 2\N$. Therefore, if $I\subset\Omega$ is a $Q$-ideal, then $\Delta I=I$ by the surjectivity of $\Delta :\Omega _+\to \Omega _+$. 
\begin{lemma}\label{Q-ideal} The following properties hold. 
\begin{itemize}
\item [{\bf (a)}] $\mathcal{Q}_0\supset\mathcal{Q}_2\supset\mathcal{Q}_4\supset\cdots$ 
\smallskip
\item [{\bf (b)}] Given $n\in 2\N$, $\mathcal{Q}_n$ is independent of the choice of $\Delta$-basis. 
\smallskip
\item [{\bf (c)}] $\Omega_r\subset (x_0,...,x_{\frac{n}{2}-1})^{r-1}+\mathcal{Q}_n$ for each integer $r\ge 1$ and $n\in 2\N$.
\end{itemize}
\end{lemma}

\begin{proof} Part (a) is clear from the definition. 

For part (b), let $\{\hat{\theta}_m\}$ be the given $\Delta$-basis, and let $\{\hat{\mu}_m\}$ be any other $\Delta$-basis for $\Omega_2$. For each $n\in 2\N$, define $Q$-ideals: 
\[
\mathcal{Q}_n=(\hat{\theta}_n,\hat{\theta}_{n+2},\hat{\theta}_{n+4},...) \quad \text{and}\quad 
\tilde{\mathcal{Q}}_n=(\hat{\mu}_n,\hat{\mu}_{n+2},\hat{\mu}_{n+4},...) 
\]
By part (a), it suffices to check $\mu _n^{(j)}\in \mathcal{Q}_n$ for each integer $j\ge 0$. By {\it Lem.~\ref{lemma1}(a)}, there exist $c_i\in k$, $0\le i\le (n+j)/2$,  such that:
\[
\mu_n^{(j)} = \sum _{0\le i\le (n+j)/2}c_i\,\theta_{2i}^{(n+j-2i)}
\]
Since $\deg_{\Delta}\mu_n^{(j)}=j$ and $\deg_{\Delta}\theta_{2i}^{(n+j-2i)}=n+j-2i$, 
and since the integers $n+j-2i$ are distinct for distinct $i$, 
it follows that $c_i=0$ when $n+j-2i > j$, i.e., when $n>2i$. 
Thus, we obtain:
\[
\mu_n^{(j)} = \sum _{n/2\le i\le (n+j)/2}c_i\,\theta_{2i}^{(n+j-2i)} \in\mathcal{Q}_n
\]
This proves part (b). 

We prove part (c) by induction on $r$, where the case $r=1$ is clear. 
Fix $n\in 2\N$ and the integer $r\ge 2$, and let $\xi\in\Omega_r$ be given. 
Observe that $\xi $ may be written as a sum of elements of $\Omega _{(r-2)}\cdot \Omega _2$. 
Since the vertices of the small $\Delta $-basis 
$\{\hat{\eta}_m\}$ form a $k$-basis for $\Omega _2$ 
by {\it Lem.~\ref{lemma1}(b)}, 
we may write 
\[
\xi = \sum_{i\ge 0}\sum_{j\ge 0}L_{(2i,j)}\eta_{2i}^{(j)}
=\sum_{i=0}^{n/2-1}\sum_{j\ge 0}L_{(2i,j)}\eta_{2i}^{(j)}
+\sum_{i\ge n/2}\sum_{j\ge 0}L_{(2i,j)}\eta_{2i}^{(j)}
\]
where $L_{(2i,j)}\in\Omega_{r-2}$. 
If $0\le i<n/2$, then 
$\eta_{2i}^{(j)}\in W_{2i}^{(j)}
\subset 
\Omega _1x_0+\cdots +\Omega _1x_{\frac{n}{2}-1}$. 
Also, by part (b) we have:
\[
\mathcal{Q}_n=(\hat{\eta}_n,\hat{\eta}_{n+2},\hat{\eta}_{n+4},...)
\]
Together, these imply 
$\xi =\xi _0x_0+\cdots +\xi _{\frac{n}{2}-1}x_{\frac{n}{2}-1}+\xi '$ 
for some $\xi _0,\ldots ,\xi _{\frac{n}{2}-1}\in \Omega _{r-1}$ 
and $\xi '\in \mathcal{Q}_n$. 
By the induction hypothesis, 
we have 
$\xi _0,\ldots ,\xi _{\frac{n}{2}-1}\in 
(x_0,...,x_{\frac{n}{2}-1})^{r-2}+\mathcal{Q}_n$. 
Therefore, 
$\xi $ belongs to 
$(x_0,...,x_{\frac{n}{2}-1})^{r-1}+\mathcal{Q}_n$. 
\end{proof}

\subsection{Cable Algebras}

\begin{definition}{\rm Let $B$ be a commutative $k$-domain. 
\begin{itemize}
\item [{\bf (a)}] $B$ is a {\it cable algebra} if there exist non-zero $D\in {\rm LND}(B)$ and a finite number of $D$-cables $\hat{s}_1,...,\hat{s}_n$ such 
that $B=A[\hat{s}_1,...,\hat{s}_n]$, where $A=\krn D$. In this case, we say that the pair $(B,D)$ is a {\it cable pair}. 
\smallskip
\item [{\bf (b)}]  $B$ is a  {\it monogenetic} cable algebra if $B=A[\hat{s}]$ for some cable pair $(B,D)$ with $A=\krn D$ and some $D$-cable $\hat{s}$. 
\smallskip
\item [{\bf (c)}] $B$ is a {\it simple} cable algebra over $k$ if $B=k[\hat{s}]$ for some $D$-cable $\hat{s}$, where $D\in {\rm LND}(B)$ is non-zero. A simple cable algebra $B$ is of {\it terminal type} if $\hat{s}$ can be chosen to be a terminal $D$-cable.
\end{itemize}
}
\end{definition}
We remark that, if there exists non-zero $D\in {\rm LND}(B)$ for which $B$ is finitely generated as an algebra over $\krn D$, then $B$ is a cable algebra. 

\begin{example} {\rm Let $B$ be a commutative $k$-domain, $D\in {\rm LND}(B)$ and $A=\krn D$. If 
\[
S\subset B\setminus (A\cup DB) \quad {\rm and}\quad  |S|=n\ge 1
\]
then there exist terminal $D$-cables 
$\hat{s}_1,...,\hat{s}_n$ such that $A[S,D]=A[\hat{s}_1,...,\hat{s}_n]$. Let $D^{\prime}$ be the restriction of $D$ to $A[S,D]$. Then $D^{\prime}\ne 0$, $A[S,D]$ is a cable algebra, 
and $(A[S,D],D^{\prime})$ is a cable pair.}
\end{example}

\begin{example} {\rm Given $n\ge 1$, let $B_n=k[x_0,...,x_n]=k^{[n+1]}$ and let $D_n$ be the restriction of the down operator to $B_n$.
The classical covariant rings $A_n=\krn D_n$ are known to be finitely generated over $k$, but have been calculated only for $n\le 8$; see \cite{Freudenburg.13}. 
Since $\partial/\partial x_n$ commutes with $D_n$, $\partial/\partial x_n$ restricts to $A_n$. If we denote this restriction by $\delta_n$, then $\krn\delta_n = A_{n-1}$. Therefore, each $A_n$ is a cable algebra. In particular:

\medskip
$A_1=k[x_0]=k^{[1]}$; see {\it Lem.~\ref{polyrings}(a)} 

\medskip
$A_2=A_1[\hat{s}]$, where $\hat{s}$ is the $\delta_2$-cable of length one with terminal vertex $s_1=2x_0x_2-x_1^2$.

\medskip
$A_3=A_2[\hat{t}]$, where $\hat{t}$ is the $\delta_3$-cable of length 2 with terminal vertex 
\[
t_2=9x_0^2x_3^2-18x_0x_1x_2x_3+6x_1^3x_3+8x_0x_2^3-3x_1^2x_2^2 .
\]

$A_4=A_3[\hat{u},\hat{v}]$, where $\hat{u},\hat{v}$ are the $\delta_4$-cables of length one with terminal vertices
\[
u_1=2x_0x_4-2x_1x_3+x_2^2 \quad {\rm and} \quad v_1 = 12x_0x_2x_4-6x_1^2x_4-9x_0x_3^2+6x_1x_2x_3-2x_2^3 .
\]
The rings $A_2,A_3,A_4$ are calculated in \cite{Freudenburg.06}, Sect. 8.6. The rings $A_5,...,A_8$ are considerably more complicated, and it would be of interest to analyze their cable structures. }
\end{example}

\subsection{Simple Cable Algebras}

A natural goal is to classify the simple cable algebras of finite transcendence degree over $k$ according to transcendence degree. 
We start with the following observation. 

\begin{lemma}\label{polyrings} 
\begin{itemize}
\item [{\bf (a)}] $k^{[1]}$ is a simple cable algebra over $k$ of non-terminal type. 
\smallskip
\item [{\bf (b)}] For each $n\ge 2$, $k^{[n]}$ is a simple cable algebra over $k$ of terminal type.
\end{itemize}
\end{lemma}
 
\begin{proof} 
Let $B=k[t]=k^{[1]}$ and let $d/dt$ denote the usual derivative. Define the sequence $t_n=\frac{1}{n!}t^n$. 
Then $\hat{t}=(t_n)$ is an infinite $d/dt$-cable and $B=k[\hat{t}]$. Therefore, $B=k^{[1]}$ is a simple cable algebra. In addition, any non-zero $D\in {\rm LND}(B)$ has a slice, so ${\rm Tr}(B,D)$ has no terminal vertices. Therefore, $B$ is of non-terminal type. This proves part (a). 

For part (b), let $B=k[x_1,...,x_n]=k^{[n]}$ and define $D$ by $Dx_{i+1}=x_i$ for $i\ge 2$ and $Dx_1=0$. Note that $x_n\not\in (DB)=(x_1,...,x_{n-1})$. Therefore, 
$\hat{x}=(x_i)$ is a terminal $D$-cable and $B=k[\hat{x}]$.
\end{proof}

Suppose that $B$ is a cable algebra with $\text{tr.deg}_kB=1$. Then $B=L^{[1]}$, where $L$ is an algebraic extension field of $k$; see \cite{Freudenburg.06}, Cor.1.24. Therefore, when $k$ is algebraically closed, $B$ is simple (over $k$) if and only if $B=k^{[1]}$. When $k$ is not algebraically closed, there are simple cable algebras over $k$ other than $k^{[1]}$. For example, consider the usual derivative $D=d/dx$ on the ring $B=\Q [\sqrt{2},x]=\Q[\sqrt{2}]^{[1]}$. We have that $\hat{s}=(\sqrt{2}\, x^n/n!)$ is  a $D$-cable and $B=\Q [\hat{s}]$, but $B\ne\Q^{[1]}$. 

For simple cable algebras of transcendence degree two, we give several illustrative examples. 

\begin{example}\label{variable} {\rm Let $B=k[x,v]=k^{[2]}$ and let $D=\partial /\partial v$. If $s_n=\frac{1}{n!}v^n$ for $n\ge 0$, then $\hat{s}=(s_n)$ is a $D$-cable rooted at 1. Let $\hat{t}=\hat{s}+_2x\hat{s}$ be given by $\hat{t}=(t_n)$. Then $B=k[\hat{t}]$, since $k[\hat{t}]$ contains $t_1=v$ and $t_2=x+\frac{1}{2}v^2$. This shows that a simple cable algebra of terminal type can also be generated by an infinite $D$-cable for some $D$. }
\end{example}

\begin{example} {\rm Continuing the notation of the preceding example, we see that 
the subring $k[x\hat{s}]$ of $k[x,v]$ is a simple cable algebra which is not finitely generated as a $k$-algebra, and therefore not of terminal type. More generally, 
let $D=\partial /\partial v$ and let $p_n(v)$ be any infinite sequence of polynomials in $k[x,v]$ with $Dp_n(v) = p_{n-1}(v)$ for $n\ge 1$ and 
$p_0(v)\in k[x]\setminus k$. 
Then $\hat{p}:=(p_n(v))$ is a $D$-cable and $k[\hat{p}]$ is a simple cable algebra of transcendence degree 2 over $k$. }
\end{example}

\begin{example} {\rm Let $B=k[y_0,y_1,y_2]$ where $2y_0y_2=y_1^2$. Define $D\in {\rm LND}(B)$ by $y_2\to y_1\to y_0\to 0$. 
It is easy to see that $y_2\not\in DB$. Therefore, $\hat{y}:=(y_n)$ is a terminal $D$-cable and $B=k[\hat{y}]$. }
\end{example}

\begin{example} {\rm The ring $B=k[z_0,z_1,z_2]$ where $2z_0^2z_2=z_1^2$ is not a simple cable algebra. In order to see this, let $D\in {\rm LND}(B)$ and a $D$-cable $\hat{s}=(s_n)$ be given. 
Define $E\in {\rm LND}(B)$ by $z_2\to z_1\to z_0^2\to 0$. It is known that ${\rm LND}(B)=k[z_0]\cdot E$ 
(see \cite{Makar-Limanov.01a}). Therefore, $DB\subset J=(z_0^2,z_1)$. Assume that $k[\hat{s}]=B$. 
If $s_n\in DB$ for every $n\ge 0$, then $B/J=k$. 
However, 
if $\pi :B\to B/J$ is the canonical surjection, 
then $\pi(z_2)$ is transcendental over $k$, 
so this case cannot occur. 
Therefore, $s_n\not\in DB$ for some $n\ge 0$, meaning that $s_n$ is a terminal vertex and $s_0,...,s_{n-1}\in J$. It follows that 
$B/J=k[\pi (s_n)]\cong k^{[1]}/(p)$ for some $p\in k^{[1]}\setminus k^*$. 
If $p=0$, then $B/J$ is an integral domain, 
a contradiction. 
If $p\ne 0$, then every element of $B/J$ is algebraic over $k$, 
a contradiction. 
Therefore, $k[\hat{s}]\ne B$.}
\end{example} 

\subsection{Cable Relations for $S$} Define the simple cable algebra $S\subset k[x,v]=k^{[2]}$ by $S=k[x\hat{s}]$, 
where $\hat{s}=(\frac{1}{n!}v^n)$. 
\begin{theorem}\label{S-relations} $S\cong_k\Omega/\mathcal{Q}_2$. Consequently, $\mathcal{Q}_2$ is a prime ideal of $\Omega$.
\end{theorem}

\begin{proof} The surjections $\phi_{\hat{s}}:\Omega\to k[v]$ and $\phi_{x\hat{s}}:\Omega\to S$ are given by:
\[
 \phi_{\hat{s}}(x_i)=s_i \quad {\rm and}\quad  \phi_{x\hat{s}}(x_i)=xs_i \quad (i\ge 0)
 \]
Let $g\in\krn\Delta$ be given, and let $\{\hat{\theta}_n\}$ be a $\Delta$-basis for $\Omega_2$. If $d/dv$ denotes the standard derivative on $k[v]$, then we have:
\begin{equation}\label{ker-Delta1}
0=\phi_{\hat{s}}\Delta g = \frac{d}{dv}\phi_{\hat{s}}g \quad\Rightarrow\quad \phi_{\hat{s}} g\in\krn\frac{d}{dv} = k 
\quad\Rightarrow\quad g\in k+\krn\phi_{\hat{s}}
\end{equation}
If $n\ge 2$ is even, then $\phi_{\hat{s}}\theta_n^{(0)}=\lambda v^n$ for some $\lambda\in k$. 
Therefore, $\theta_n^{(0)}\in\krn\phi_{\hat{s}}$ for each even $n\ge 2$. 

Given an even integer $n\ge 2$, assume that $\theta_n^{(j)}\in\krn\phi_{\hat{s}}$ for some $j\ge 0$. We have:
\[
0=\phi_{\hat{s}}\theta_n^{(j)}=\phi_{\hat{s}}\Delta\theta_n^{(j+1)} = \frac{d}{dv}\phi_{\hat{s}}\theta_n^{(j+1)}
 \quad\Rightarrow\quad \phi_{\hat{s}} \theta_n^{(j+1)}\in\krn\frac{d}{dv} = k 
\]
As before, since $n\ge 2$, we must have $\phi_{\hat{s}} \theta_n^{(j+1)}=0$. It follows by induction that 
$\theta_n^{(j)}\in\krn\phi_{\hat{s}}$ for every even $n\ge 2$ and every $j\ge 0$. 
Therefore, $\mathcal{Q}_2\subset\krn\phi_{\hat{s}}$. 

Given $r\ge 2$ and $P\in\Omega_r$, note that $\phi_{x\hat{s}}P=x^r\phi_{\hat{s}}P$. Therefore, if $P\in\Omega$ is homogeneous, then $P\in\krn\phi_{x\hat{s}}$ if and only if 
$P\in\krn\phi_{\hat{s}}$. In particular, this implies $\mathcal{Q}_2\subset\krn\phi_{x\hat{s}}$. 

Suppose that $P\in\Omega_r\cap\krn\phi_{x\hat{s}}$. 
By {\it Lem.~\ref{Q-ideal}(c)}, we see that $P\in (x_0)^{r-1}+\mathcal{Q}_2$. 
Write $P=x_0^{r-1}L+Q$ for $L\in\Omega$ and $Q\in\mathcal{Q}_2$. Since the element $P$  and the ideals $(x_0)^{r-1}$ and $\mathcal{Q}_2$ are homogeneous, we may assume that $L$ and $Q$ are homogeneous. By degree considerations, $L\in\Omega_1$. We have that 
$x_0^{r-1}L\in\krn\phi_{x\hat{s}}$. If $L\ne 0$, then since $\krn\phi_{x\hat{s}}$ is a prime ideal, either $x_0\in\krn\phi_{x\hat{s}}$ or $L\in\krn\phi_{x\hat{s}}$, a contradiction. Therefore, $L=0$ and $P\in\mathcal{Q}_2$. 

We have thus shown $\Omega_r\cap\krn\phi_{x\hat{s}}\subset\mathcal{Q}_2$ for all $r\ge 2$. This suffices to prove 
$\krn\phi_{x\hat{s}}=\mathcal{Q}_2$.
\end{proof}


\section{The Derivation $D$ in Dimension Five}\label{Five}

\subsection{Definitions}
Define the polynomial ring $B=k[a,x,y,z,v]=k^{[5]}$. We define the locally nilpotent derivation $D$ of $B$ by its action on a set of generators:
\[
z\to y\to x\to a^3 \quad ,\quad v\to a^2\quad ,\quad a\to 0
\]
Define $A=\krn D$ and $R=k[a,x,y,z]$, 
noting that $D$ restricts to $R$. In fact, $D$ restricts to a linear derivation of the subring $k[a^3,x,y,z]$, and this kernel is well known. 
Let $k[t,x,y,z]=k^{[4]}$ and define the linear derivation $\tilde{D}$ on this ring by $z\to y\to x\to t\to 0$. Then $\krn\tilde{D}=k[t,\tilde{F},\tilde{G},\tilde{h}]$, where: 
\[
\tilde{F}=2ty-x^2 \,\, ,\,\, \tilde{G}=3t^2z-3txy+x^3\quad\text{and}\quad  t^2\tilde{h}=\tilde{F}^3+\tilde{G}^2
\]
See \cite{Freudenburg.06}, Example 8.9.
Note that the restriction of $D$ to $R$ is equal to 
the $k[a]$-derivation ${\rm id}_{k[a]}\otimes \tilde{D}$ 
on $k[a]\otimes _{k[t]}k[t,x,y,z]=R$, 
and its kernel $R\cap A$ is equal to 
$\ker ({\rm id}_{k[a]}\otimes \tilde{D})
=k[a]\otimes _{k[t]}\ker \tilde{D}$. 
Therefore, if $F=\tilde{F}\vert_{t=a^3}$, $G=\tilde{G}\vert_{t=a^3}$ and $h=\tilde{h}\vert_{t=a^3}$, then:
\[
R\cap A=k[a,F,G,h] \quad {\rm where}\quad a^6h=F^3+G^2
\]
Specifically:
\[
F=2a^3y-x^2\,\, ,\,\, G=3a^6z-3a^3xy+x^3\,\, ,\,\,  h=9a^6z^2-18a^3xyz+8a^3y^3+6x^3z-3x^2y^2
\]
Define a $\Z^2$-grading of $B$ by declaring that $a,x,y,z,v$ are homogeneous and: 
\[
\deg (a,x,y,z,v)=((1,0),(3,1),(3,2),(3,3),(2,1))
\]
Then $D$ is a homogeneous derivation of degree $(0,-1)$ and $A$ is a graded subring of $B$. 
Given integers $r,s\ge 0$, let $B_{(r,s)}$ be the vector space of homogeneous polynomials in $B$ of degree $(r,s)$, and define:
\[
A_{(r,s)}=A\cap B_{(r,s)}
\]
Then we have:
\[
F\in A_{(6,2)}\,\, ,\,\, G\in A_{(9,3)}\,\, ,\,\, h\in A_{(12,6)}
\]
Since $k[a,F,G,h]=R\cap A=\ker D|_R$ is factorially closed in $R$, we see that $F$, $G$ and $h$ are irreducible by degree considerations. 
Note that $[D,\partial /\partial v]=0$, that is, $D$ commutes with the partial derivative $\partial /\partial v$ on $B$. Therefore, $\partial /\partial v$ restricts to $A$. If $\partial$ denotes the restriction of $\partial /\partial v$ to $A$, then $\partial\in {\rm LND}(A)$ and $\partial$ is homogeneous of degree $(-2,-1)$. 

The following result is needed below. 

\begin{lemma}\label{mini} Given $n\ge 0$, write $n=6e+\ell$ for $e\ge 0$ and $0\le \ell\le 5$.
\begin{itemize}
\item [{\bf (a)}] 
\[
R\cap A_{(2n+1,n)}=\begin{cases} \langle ah^e\rangle & \ell =0 \\ \{ 0\} & \ell\ne 0\end{cases}
\]
\item [{\bf (b)}] 
\[
R\cap A_{(2n+2,n)} = \begin{cases} \langle a^2h^e\rangle & \ell =0\\ \langle Fh^e\rangle & \ell =2\\ \{ 0\} & \ell =1,3,4,5\end{cases}
\]
\end{itemize}
\end{lemma}

\begin{proof} Since $R\cap A= k[a,F,G,h]$ with $a,F,G,$ and $h$ homogeneous, each $k$-vector space $R\cap A_{(r,s)}$ is spanned by monomials in $a,F,G$ and $h$. 
If the monomial $a^{e_1}F^{e_2}G^{e_3}h^{e_4}\in R$ ($e_i\in\N$) has degree $(2n+1,n)$, then:
\[
\begin{cases} e_1+6e_2+9e_3+12e_4 &=2n+1\\ 2e_2+3e_3+6e_4 &=n \end{cases}
\]
The solutions to this system are $e_1=1$, $e_2=e_3=0$ and $6e_4=n$. This proves part (a). 

Similarly, if $\deg (a^{e_1}F^{e_2}G^{e_3}h^{e_4})=(2n+2,n)$, then:
\[
\begin{cases} e_1+6e_2+9e_3+12e_4 &= 2n+2\\ 2e_2+3e_3+6e_4 &= n\end{cases}
\]
The solutions to this system are:
\[
\{ e_1=2\,\, ,\,\, e_2=e_3=0\,\, ,\,\, n=6e_4\}\quad {\rm and}\quad 
\{ e_1=e_3=0\,\, ,\,\, e_2=1\,\, ,\,\, n=6e_4+2\}
\]
This proves part (b). 
\end{proof}

\subsection{Homogeneous $\partial$-Cables} Let $\mathcal{S}_a$ denote the set of infinite homogeneous $\partial$-cables rooted at $a$.

\begin{theorem}\label{cable} $\mathcal{S}_a\ne\emptyset$
\end{theorem}

\begin{proof} We show that there exists a sequence $s_n\in A$, $n\ge 0$, such that:
\begin{itemize}
\item [(a)]  $s_0=a$
\smallskip
\item [(b)] $s_n\in A_{(2n+1,n)}$ for each $n\ge 0$
\smallskip
\item [(c)]  $\partial s_n=s_{n-1}$  for each $n\ge 1$
\end{itemize}

Let $d$ denote the restriction of $D$ to the subring $Q\subset B$ defined by $Q=k[t,x,y,z]\cong k^{[4]}$, where $t=a^3$. Then $d$ is a linear derivation defined by:
\[
z\to y\to x\to t\to 0
\]
In addition, $d$ is homogeneous of degree $(0,-1)$ for the $\Z^2$-grading of $Q$ for which:
\[
\deg (t,x,y,z)=((1,0), (1,1), (1,2),(1,3))
\]
Let $Q_{(r,s)}$ denote the vector space of homogeneous polynomials in $Q$ of degree $(r,s)$. Then according to Proposition 4.1 of \cite{Freudenburg.13}, the mapping
\[
d: Q_{(r,s+1)}\to Q_{(r,s)}
\]
is surjective if $2s<3r$.
Thus, given $m\ge 1$, each mapping in the following sequences of maps is surjective:
\[
t\cdot Q_{(2m,3m)}\subset Q_{(2m+1,3m)}\xleftarrow{d} Q_{(2m+1,3m+1)}\xleftarrow{d} Q_{(2m+1,3m+2)}
\]
and
\[
t\cdot Q_{(2m-1,3m-1)}\subset Q_{(2m,3m-1)}\xleftarrow{d} Q_{(2m,3m)}
\]
Consequently, there exists a sequence $w_n\in Q$, $n\ge 0$, such that $w_0=1$, and for all $m\ge 0$,
\[
w_{3m}\in Q_{(2m,3m)}\,\, ,\,\, w_{3m+1}\in Q_{(2m+1,3m+1)}\,\, ,\,\, w_{3m+2}\in Q_{(2m+1,3m+2)}
\]
where:
\[
dw_{3m+3}=t\cdot w_{3m+2}\quad ,\quad dw_{3m+2}=w_{3m+1}\quad ,\quad dw_{3m+1}=t\cdot w_{3m}
\]
With the sequence $w_n$ so constructed, it follows that for $m\ge 1$:
\begin{equation*}
D^{3i}w_{3m}=d^{3i}w_{3m}=t^{2i}w_{3(m-i)}=a^{6i}w_{3(m-i)}=(Dv)^{3i}w_{3(m-i)}\quad (0\le i\le m)
\end{equation*}
Therefore, for $0\le i\le m$, we have:
\medskip
\begin{itemize}
\item [(i)] $D^{3i}(aw_{3m}) = a(Dv)^{3i}w_{3(m-i)}$
\smallskip
\item [(ii)] $D^{3i+1}(aw_{3m}) = d(a(Dv)^{3i}w_{3(m-i)}) = a(Dv)^{3i}tw_{3(m-i)-1} = a^2(Dv)^{3i+1}w_{3(m-i)-1}$
\smallskip
\item [(iii)] $D^{3i+2}(aw_{3m}) = d(a^2(Dv)^{3i+1}w_{3(m-i)-1}) = a^2(Dv)^{3i+1}w_{3(m-i)-2} = (Dv)^{3i+2}w_{3(m-i)-2}$
\end{itemize}
\medskip
We see that:
\begin{equation}\label{divides}
 (Dv)^j \,\, {\rm divides}\,\, D^j(aw_{3m}) \,\, {\rm for\,\, each}\,\, j \,\, (0\le j\le 3m)
\end{equation}
Therefore, if we define $s_{3m}= (-1)^{3m}\pi_v(aw_{3m})$ for $m\ge 0$, then $s_{3m}\in A$ for each $m\ge 0$. 
Using the equality (\ref{commute}) in {\it Sect.~2.1}, it follows that for $m\ge 1$:
\begin{eqnarray*}
\frac{\partial^3}{\partial v^3}s_{3m} &=& \frac{\partial^2}{\partial v^2} (-1)^{3m-1}\pi_v(aDw_{3m})\frac{\partial}{\partial v}\frac{v}{a^2}\\
&=& \frac{\partial}{\partial v} (-1)^{3m-2}\pi_v(aD^2w_{3m})\frac{1}{a^2}\frac{\partial}{\partial v}\frac{v}{a^2}\\
&=&  (-1)^{3m-3}\pi_v(aD^3w_{3m})\frac{1}{a^4}\frac{\partial}{\partial v}\frac{v}{a^2}\\
&=& (-1)^{3m-3}\pi_v(a(a^2)^3w_{3(m-1)})\frac{1}{a^6}\\
&=& (-1)^{3(m-1)}\pi_v(aw_{3(m-1)})\\
&=& s_{3(m-1)}
\end{eqnarray*}
Define:
\[
s_{3m-1}=\frac{\partial}{\partial v}s_{3m} \quad {\rm and}\quad s_{3m-2}=\frac{\partial}{\partial v}s_{3m-1}\quad (m\ge 1)
\]
Then $\hat{s}:=(s_n)$ is a $\partial$-cable rooted at $a$ with $s_n\in A_{(2n+1,n)}$ for each $n\ge 0$.
\end{proof}

\begin{remark}\label{list} {\rm Let $\hat{s}=(s_n)\in \mathcal{S}_a$ be given. Since $\dim A_{(2n+1,n)}=1$ for $n=0,...,5$, the elements $s_0,...,s_5$ are uniquely determined; see {\it Cor.~\ref{dimension1} (a)}. They are given by:
\begin{eqnarray*}
0!s_0&=&a\\
1!s_1&=&av-x\\
2!s_2&=&av^2-2xv+2a^2y\\
3!s_3&=&av^3-3xv^2+6a^2yv-6a^4z\\
4!s_4&=&av^4-4xv^3+12a^2yv^2-24a^4zv+24a^3xz-12a^3y^2\\
5!s_5&=&av^5-5xv^4+20a^2yv^3-60a^4zv^2+120a^3xzv-60a^3y^2v-72x^2a^2z+36xa^2y^2+24a^5yz
\end{eqnarray*}
Note the identities:}
\begin{equation}\label{FG}
F=2s_0s_2-s_1^2 \,\, ,\,\, -G = 3s_0^2s_3-3s_0s_1s_2+s_1^3\,\, ,\,\, 2s_0s_4=2s_1s_3-s_2^2 \,\, ,\,\,  5s_0s_5=3s_1s_4-s_2s_3 
\end{equation}
\end{remark}


\section{Generators of $\bar{A}$ and $A$}

The main result of this section is the following.

\begin{theorem}\label{A-generators} Let $\hat{s}=(s_n)\in \mathcal{S}_a$ be given. 
\begin{itemize}
\item [{\bf (a)}] $A=k[h,\hat{s}]$
\smallskip
\item [{\bf (b)}] $A$ is not finitely generated as a $k$-algebra
\smallskip
\item [{\bf (c)}] The generating set $\{ h,s_n\}_{n\ge 0}$ is minimal in the sense that no proper subset generates $A$. 
\end{itemize}
\end{theorem}

\subsection{Generators of $\bar{A}$}
Let $\pi :B\to B/hB$ be the canonical surjection. Given $b\in B$, let $\bar{b}$ denote $\pi (b)$, and for a subalgebra $M\subset B$, let $\bar{M}=\pi (M)$. 
Since $h$ is homogeneous, $\pi$ induces a $\Z^2$-grading on $\bar{B}$, and $\bar{A}$ is a graded subring with:
\[
\bar{A}_{(r,s)}=\pi (A_{(r,s)})
\]
Note that, 
since $h$ is irreducible, $hB$ is a prime ideal of $B$. 
Hence, $B/hB$ and its subring $\bar{A}$ are integral domains. 
Since $D(h)=0$, we have $hB\cap A=hA$. 
Indeed, 
if $P\in B$ is such that $hP\in A$, 
then $hDP=D(hP)=0$, and hence $DP=0$. 
Thus, 
$\bar{A}\cong A/hA$ and so $hA$ is a prime ideal of $A$. 
Since $h\in\krn\partial$, we can define $\delta\in {\rm LND}(\bar{A})$ by $\delta\pi (g)=\pi\partial (g)$. Then $\delta$ is a homogeneous locally nilpotent derivation of $\bar{A}$ of degree $(-2,-1)$. 
Recall that $\krn\partial = R\cap A = k[a,F,G,h]$. 

\begin{lemma}\label{ker-delta} $\krn\delta = \pi(\krn\partial )=k[\bar{a},\bar{F},\bar{G}]$
\end{lemma}

\begin{proof} It must be shown that $\partial^{-1}(hA)=R\cap A+hA$. The inclusion $R\cap A+hA\subset \partial^{-1}(hA)$ is clear. For the converse, we first show that, if $H=R\cap A+hB$, then $H\cap aB=aH$.

Since $R\cap A=k[a,F,G,h]$ and $F^3+G^2\in hR$, we have:
\[
H=k[a,F]+k[a,F]G+hB
\]
In addition, $H$ is a graded subring of $B$, and if $g\in H_{(r,s)}$, then $g\in k[a,F]+hB$ for $s$ even, and $g\in k[a,F]G+hB$ for $s$ odd. Write 
$g=p(a,F)G^{\epsilon}+h\rho$, where $p\in k^{[2]}$, $\rho\in B$ and $\epsilon\in\{ 0,1\}$. If $g\in aB$, then setting $a=0$ yields the following equation in $k[x,y,z,v]$:
\[
(h\rho )|_{a=0}=3x^2(2xz-y^2)\rho |_{a=0}   =-p(0,-x^2)x^{3\epsilon}\in k[x]
\]
This means $\rho\in aB$, since $2xz-y^2$ is transcendental over $k[x]$. Therefore, $p(a,F)\in aB$, and since $R\cap A$ is factorially closed in 
$B$ it follows that $p(a,F)\in a(R\cap A)$. So $g\in aH$. This shows $H\cap aB=aH$.

Suppose that $f\in A$ and $\partial f\in hA$. Let $L\in R^{[1]}$ be such that $f=L(v)=\sum_i\frac{1}{i!}L^{(i)}(0)v^i$. 
We have:
\[
\partial^if=L^{(i)}(v)\in hA \quad \forall i\ge 1
\quad\Rightarrow\quad L^{(i)}(0)\in hR \quad \forall i\ge 1
\]
Therefore, $f=hq+r$ for $q\in B$ and $r=L(0)\in R$. It follows that $0=Df=hDq+Dr$, which implies $Dr\in R\cap hB=hR$. 

The restriction of $D$ to $R$ has kernel $R\cap A$ and local slice $x$. 
So there exists $n\ge 0$ and $P\in (R\cap A)^{[1]}$ with $a^nr=P(x)=\sum_i\frac{1}{i!}P^{(i)}(0)x^i$.
We thus have:
\[
a^nD^ir=P^{(i)}(x)a^{3i}\in hR \quad \forall i\ge 1 \quad\Rightarrow\quad P^{(i)}(x)\in hR \quad \forall i\ge 1
\quad\Rightarrow\quad P^{(i)}(0)\in h(R\cap A) \quad \forall i\ge 1
\]
Therefore, $a^nr\in (R\cap A) + h(R\cap A)[x]\subset H$. 
 
By repeated application of the identity $H\cap aB = aH$, we have that $H\cap a^nB=a^nH$. It follows that 
$a^nr\in H\cap a^nB=a^nH$. 
Therefore, $r\in H$ and $f=hq+r\in hB+H = H$. Since $A$ is factorially closed in $B$, we conclude that $f\in R\cap A+hA$.
\end{proof} 

Given $\hat{s}=(s_n)\in \mathcal{S}_a$, 
we have $s_0=a\not\in hB$, and so $\bar{s}_0\ne 0$. 
Since $\delta \pi =\pi \partial $, 
we see that $\pi \hat{s}:=(\bar{s}_n)$ 
is a $\delta $-cable. 
If $\phi_{\pi\hat{s}}:\Omega\to \bar{A}$ 
is the associated mapping, then 
$\phi_{\pi\hat{s}}\Delta = \delta\phi_{\pi\hat{s}}$ 
(c.f.~{\it Sect.~\ref{D-cables}} (x)). 
We also note that $\krn \phi_{\pi\hat{s}}$ 
is a homogeneous ideal of $\Omega $, 
since 
$\phi _{\pi\hat{s}}(\Omega _{(r,s)})\subset \bar{A}_{(2s+r,s)}$ 
for each $r,s\ge 0$.

\begin{theorem}\label{surjective} $\phi_{\pi\hat{s}}$ is surjective.
\end{theorem}

\begin{proof}
Define:
\[
A' =\phi_{\pi\hat{s}}(\Omega ) = k[\pi\hat{s}]  \,\, , \quad A'_+ = \phi_{\pi\hat{s}} (\Omega_+) \quad {\rm and}\quad A'_{(r,s)} = A'\cap\bar{A}_{(r,s)}
\]
Since $\Delta :\Omega_+\to\Omega_+$ is surjective and $\phi_{\pi\hat{s}}\Delta = \delta\phi_{\pi\hat{s}}$, it follows that the mapping $\delta :A'_+\to A'_+$ is surjective. 
In addition, define:
\[
C=\krn\delta \quad {\rm and}\quad C_{(r,s)}=C\cap\bar{A}_{(r,s)}
\]
Then from {\it Lem.~\ref{ker-delta}} and (\ref{FG}) we see that:
\begin{equation}\label{FG-bar}
C=k[\bar{a},\bar{F},\bar{G}] \,\, ,\quad \bar{F}=2\bar{s}_0\bar{s}_2-\bar{s}_1^2 \,\, ,\quad  -\bar{G}=3\bar{s}_0^2\bar{s}_3-3\bar{s}_0\bar{s}_1\bar{s}_2+\bar{s}_1^3
\end{equation}
Therefore, $C\subset A'$ and $\krn\delta\vert_{A'}=C$.

Fix $\ell\in\Z$. We show by induction on $n$ that, for each integer $n\ge 0$:
\begin{equation}\label{induct}
A'_{(2n+\ell ,n)} = \bar{A}_{(2n+\ell ,n)}
\end{equation}
For $n=0$, it is easy to see that $\bar{A}_{(\ell ,0)}=\{ 0\}$ if $\ell <0$. 
If $\ell\ge 0$, 
then $\bar{A}_{(\ell ,0)}=\langle a^{\ell }\rangle 
=\langle\bar{s}_0^{\ell}\rangle$, since $B_{(\ell ,0)}=\langle a^{\ell }\rangle $. 
So (\ref{induct}) holds for $n=0$. 
Since $B_{(2,1)}=\langle v\rangle $, 
we have $A'_{(2,1)}=\bar{A}_{(2,1)}=\{ 0\} $. 
Hence, (\ref{induct}) also holds for $n=1$ and $\ell =0$.

Given $n\ge 1$, assume that: 
\[
(n,\ell )\ne (1,0)\quad\text{and}\quad
A'_{(2(n-1)+\ell ,n-1)}=\bar{A}_{(2(n-1)+\ell ,n-1)}
\]
Since $\delta :A'_+\to A'_+$ is surjective and $A'_+ = \oplus_{(r,s)\ne (0,0)}A'_{(r,s)}$,
it follows that:
\[
\delta A'_{(2n+\ell ,n)}=A'_{(2(n-1)+\ell ,n-1)}=\bar{A}_{(2(n-1)+\ell ,n-1)}
\]
Since $A'_{(2n+\ell ,n)}\subset\bar{A}_{(2n+\ell ,n)}$, we have
\[
 \bar{A}_{(2(n-1)+\ell,n-1)}=\delta A'_{(2n+\ell ,n)}\subset\delta\bar{A}_{(2n+\ell ,n)}\subset \bar{A}_{(2(n-1)+\ell ,n-1)}
 \]
which implies $\delta A'_{(2n+\ell ,n)} = \delta\bar{A}_{(2n+\ell ,n)}$. 
Therefore:
\begin{eqnarray*}
\dim \bar{A}_{(2n+\ell ,n)} &=& \dim C_{(2n+\ell ,n)} + \dim\delta\bar{A}_{(2n+\ell ,n)} \\
&=& \dim C_{(2n+\ell ,n)} + \dim\delta A'_{(2n+\ell ,n)} \\
&=& \dim A'_{(2n+\ell ,n)} 
\end{eqnarray*}
It follows that $A'_{(2n+\ell ,n)} = \bar{A}_{(2n+\ell ,n)}$. 
By induction, we conclude that (\ref{induct}) holds for all $n\ge 0$. 
\end{proof}

\begin{corollary}\label{cor1} Let $\hat{s}=(s_n)\in \mathcal{S}_a$ be given. 
\begin{itemize}
\item [{\bf (a)}] $\bar{A}=k[\pi\hat{s}]$
\smallskip
\item [{\bf (b)}] $\bar{A}$ is not finitely generated as a $k$-algebra
\smallskip
\item [{\bf (c)}] The generating set $\{ \bar{s}_n\}_{n\ge 0}$ is minimal in the sense that no proper subset generates $\bar{A}$. 
\end{itemize}
\end{corollary}

\begin{proof} Part (a) is implied by {\it Thm.~\ref{surjective}}. For part (b), let $\Sigma\subset\N^2$ be the degree semi-group of $A$. Then part (a) implies that:
\[
\Sigma = \langle (2n+1,n)\,\vert\, n\ge 0\rangle
\]
It will suffice to show that $\Sigma$ is not finitely generated as a semi-group. However, this is obvious, since the element $(2n+1,n)$ does not belong to the sub-semigroup generated by 
$(2i+1,i)$ for $i<n$. This proves part (b). In fact, $(2n+1,n)$ does not even belong to the larger sub-semigroup generated by $(2i+1,i)$ for $i\ne n$, and this implies part (c). 
\end{proof}

\subsection{Proof of Thm.~\ref{A-generators}}

Set $\Gamma =k[\hat{s}]$. Then $\Gamma$ is a graded subring of $A$, where $\Gamma_{(r,s)}=\Gamma\cap A_{(r,s)}$. 
By {\it Cor.~\ref{cor1}(a)}, each $g\in A$ has the form $g=\gamma +h\cdot\alpha$, where $\gamma\in\Gamma$ and $\alpha\in B$. 
Since $g,\gamma ,h\in A$, it follows that $\alpha \in A$. 
Write
\[
\gamma = \sum \gamma_{(r,s)} \quad {\rm and}\quad \alpha = \sum\alpha_{(r,s)}
\]
where $\gamma_{(r,s)}\in \Gamma_{(r,s)}$ and $\alpha_{(r,s)}\in A_{(r,s)}$ for each $r,s\in \Z $. Then the homogeneous decomposition of $g$ is:
\[
g = \sum_{(r,s)}\left(\gamma_{(r,s)}+h\cdot\alpha_{(r-12,s-6)}\right)
\]
When $g$ is homogeneous, there exists $(r,s)$ such that $g=\gamma_{(r,s)}+h\cdot\alpha_{(r-12,s-6)}$. 

For each fixed $r\ge 0$, we show by induction on $s$ that $A_{(r,s)}\subset\Gamma [h]$. We have $A_{(r,0)}=k\cdot a^r\subset \Gamma$, which gives the basis for induction. Given $s\ge 1$, suppose that $A_{(r,i)}\subset\Gamma [h]$ for $0\le i\le s-1$. Given $g\in A_{(r,s)}$, write $g=\gamma_{(r,s)}+h\cdot\alpha_{(r-12,s-6)}$ as above. 
By the induction hypothesis, we have that $\alpha_{(r-12,s-6)}\in \Gamma [h]$. Therefore $g\in\Gamma [h]$. We conclude that $A_{(r,s)}\subset\Gamma [h]$ for all $(r,s)$ with $r,s\ge 0$, and therefore $A\subset\Gamma [h]$. This proves part (a). 

Part (b) is immediately implied by {\it Cor.~\ref{cor1}(b)} and the fact that $\bar{A}$ is the image of $A$ under a $k$-algebra homomorphism. 

For part (c), note that {\it Cor.~\ref{cor1}(c)} implies that any generating subset of $\{ h,s_n\}_{n\ge 0}$ must include each $s_n$. We also cannot exclude $h$, since $(12, 6)$ does not belong to the degree semigroup generated by $\{ (2n+1,n)\,\vert\, n\ge 0\}$. This proves part (c) and completes the proof of {\it Thm.~\ref{A-generators}}. 
$\Box$
\medskip

For the next result, the reader is reminded that $A_{(r,s)}=\{ 0\}$ if $r<0$ or $s<0$. 

\begin{corollary}\label{dimension1} Let $\hat{s}=(s_n)\in \mathcal{S}_a$. Given $n\ge 0$, let $e\ge 0$ be such that $0\le n-6e\le 5$. 
\begin{itemize}
\item [{\bf (a)}] $A_{(2n+1,n)} = k\cdot s_n\oplus h\cdot A_{(2(n-6)+1,n-6)}$ 
\medskip
\item [{\bf (b)}] $\dim A_{(2n+1,n)} = e+1$ 
\medskip
\item [{\bf (c)}] A basis for $A_{(2n+1,n)}$ is given by $\{ s_n,s_{n-6}h,s_{n-12}h^2,...,s_{n-6e}h^e\}$. 
\end{itemize}
\end{corollary}

\begin{proof} 
Part (a) is implicit in the first paragraph of 
the proof of {\it Thm.~\ref{A-generators}} 
with $(r,s)=(2n+1,n)$, 
since $\Gamma _{(2n+1,n)}=k\cdot s_n$ 
and $s_n\not\in hB$. 
It follows that $A_{(2n+1,n)}=k\cdot s_n$ 
for $n=0,\ldots ,5$. 
Therefore, 
using part (a), we get parts (b) and (c) by induction on $n$. 
\end{proof}

\begin{remark}\label{infinite} {\rm Consider the field $k(h)=k^{(1)}$ and the $k(h)$-algebra $k(h)\otimes_{k[h]}A=k(h)[\hat{s}]$. Since $\partial h=0$, $\partial$ extends to a locally nilpotent derivation $\tilde{\partial}$ of $k(h)[\hat{s}]$, $\hat{s}$ is a $\tilde{\partial}$-cable, and $k(h)[\hat{s}]$ is a simple cable algebra over $k(h)$ which is of transcendence degree 3 over $k(h)$.}
\end{remark}

\subsection{The $\partial$-Cable $\hat{\sigma}$}

\begin{theorem}\label{unique} There exists a unique $\hat{\sigma}=(\sigma_n)\in \mathcal{S}_a$ such that $n!\sigma_n\equiv -nxv^{n-1}\, ({\rm mod}\, aB)$ for each $n\ge 1$. In addition, $\hat{\sigma}$ satisfies the following.
\begin{itemize}
\item [{\bf (a)}]  If $n,e\ge 0$ with $n\ne 1$, 
then: $\sigma_0\sigma_1h^e\not\in 
\langle \sigma _i\sigma _{n-i}\mid 0\le i\le n/2,i\ne 1\rangle $
\medskip
\item [{\bf (b)}] If $n,e\ge 0$ with $n\ne 2$, 
then: $Fh^e\not\in 
\langle \sigma _i\sigma _{n-i}\mid 0\le i\le n/2\rangle $
\end{itemize}
\end{theorem}

\begin{proof} Given $P\in B$, let $P(0)$ denote evaluation at $v=0$. 

An explicit sequence $w_n\in k[t,x,y,z]$ of the type used in the proof of {\it Thm.~\ref{cable}} is constructed in \cite{Freudenburg.06}, \S{\,}7.2.1, and in this example, $w_n$ has the property:
\begin{quote}\begin{center}
$t$ divides $w_n$ whenever $n\ge 4$ and $n\equiv 1\, ({\rm mod}\, 3)$
\end{center}\end{quote}
Let $\hat{\sigma}=(\sigma_n)\in \mathcal{S}_a$ be the $\partial$-cable constructed from this sequence. Given $m\ge 1$, it follows from the definition of the functions $s_n=\sigma_n$ given in the proof of {\it Thm.~\ref{cable}} that:
\begin{align*}
\sigma _{3m}
=(-1)^{3m}aw_{3m}-D\left((-1)^{3m}aw_{3m}\right)\dfrac{v}{a^2}
+\dfrac{1}{2}D^2\left((-1)^{3m}aw_{3m}\right)\dfrac{v^2}{a^4}
+\cdots 
\end{align*}
Since $\partial^i\sigma _{3m}/\partial v^i=\sigma _{3m-i}$ 
for $0\le i\le 3m$, this implies that: 
\[
\sigma_{3m}(0)=(-1)^{3m}aw_{3m}\,\, ,\,\, \sigma_{3m-1}(0)=(-1)^{3m-1}a^2w_{3m-1}\,\, ,\,\, \sigma_{3m-2}(0)=(-1)^{3m-2}w_{3m-2}
\]
Since $t=a^3$ divides $w_{3m-2}$ for $m\ge 2$ 
and $\sigma _0(0)=\sigma _0=a$, 
it follows that $a$ divides $\sigma_n(0)$ 
for all $n\ge 0$ with $n\ne 1$. 
We now show by induction on $n$ that:
\begin{equation}\label{hypoth}
a \,\, {\rm divides}\,\, P_n(v):=(n-1)!\sigma_n+xv^{n-1} \quad (n\ge 1)
\end{equation}

First, observe that {\it Cor.~\ref{dimension1}(b)} implies that the functions $\sigma_0,...,\sigma_5$ are uniquely determined. 
In particular, we have 
$\sigma _1=av-x$ (see {\it Remark~\ref{list}}). 
Hence, 
property (\ref{hypoth}) holds for $n=1$. 

Given $n\ge 2$, assume that $a$ divides $P_i(v)$ for $1\le i\le n-1$. We have:
\[
P_n^{\prime}(v) =(n-1)!\sigma_{n-1}+(n-1)xv^{n-2} = (n-1)P_{n-1}(v)
\]
The inductive hypothesis implies that $P_n^{\prime}(v)\in aB$, which means $P_n(v)-P_n(0)\in aB$. Since $P_n(0)=(n-1)!\sigma_n(0)\in aB$, we conclude that 
$P_n(v)\in aB$ for all $n\ge 1$. This proves the existence of $\hat{\sigma}=(\sigma_n)\in\mathcal{S}_a$ such that $n!\sigma_n\equiv -nxv^{n-1}\, ({\rm mod}\, aB)$.

For uniqueness, let $\hat{s}=(s_n)\in\mathcal{S}_a$ be such that $n!s_n\equiv -nxv^{n-1}\, ({\rm mod}\, aB)$ for $n\ge 1$.
Choose $N\ge 1$ such that 6 does not divide $N$, and let $e\ge 0$ be such that $1\le N-6e\le 5$. By {\it Cor.~\ref{dimension1}(c)}, a  basis for $A_{(2N+1,N)}$ is given by:
\[
s_N', s_{N-6}'h,s_{N-12}'h^2,\ldots ,s_{N-6e}'h^e,\ \text{where}\ 
s_n':=n!s_n
\]
Therefore, there exist $c_i\in k$ with 
$N!\sigma_N=c_0s_N'+c_1s_{N-6}'h+\cdots +c_es_{N-6e}'h^e$. The substitution $a\mapsto 0$ yields: 
\[
-Nxv^{N-1} = -c_0Nxv^{N-1}-c_1(N-6)xv^{N-7}h'-\cdots - c_e(N-6e)xv^{N-6e-1}(h')^e,
\]
where $h'=3x^2(2xz-y^2)$. 
This implies that $c_0=1$ and $c_1=\cdots = c_e=0$, meaning that $\sigma_N=s_N$. Therefore, $\hat{\sigma}$ and $\hat{s}$ agree on an infinite number of vertices, which implies that $\hat{\sigma}=\hat{s}$; see {\it \S\ref{D-cables} (vi)}. This proves the uniqueness assertion. 

In order to prove properties (a) and (b), 
recall that $\sigma _i(0)\in aB$ for all $i\ge 0$ with $i\ne 1$. 
Hence, 
$\sigma _i(0)\sigma _{n-i}(0)\in a^2B$ 
($0\le i\le n/2$, $i\ne 1$) if $n\ne 1$, 
and 
$\sigma _i(0)\sigma _{n-i}(0)\in aB$ 
($0\le i\le n/2$) if $n\ne 2$. 
To show (a), 
suppose that 
$\sigma_0\sigma_1h^e\in 
\langle \sigma _i\sigma _{n-i}\mid 0\le i\le n/2,i\ne 1\rangle $. 
Then, 
we have: 
$$
-axh^e=(\sigma_0\sigma_1h^e)|_{v=0}\in 
\langle \sigma _i(0)\sigma _{n-i}(0)
\mid 0\le i\le n/2,i\ne 1\rangle 
\subset a^2B, 
$$
and so $xh^e\in aB$, a contradiction. 
Since $Fh^e\in R\setminus aB$, 
property (b) is proved similarly. 
\end{proof}

We remark that {\it Thm.~\ref{unique}(b)}, 
together with {\it Lem.~\ref{mini}(b)}, implies 
$R\cap A_{(2n+2,n)}\cap \phi _{\hat{\sigma }}(\Omega _{(2,n)})
=\{ 0\} $ 
if $n\equiv 2\pmod{6}$ and $n\ne 2$.

\begin{corollary}\label{prime-ideal} Let $S\subset k[x,v]=k^{[2]}$ be the subalgebra $S=k[x,xv,xv^2,\ldots ]$. 
Given $\lambda\in k$, put $J_{\lambda}=aA+(h-\lambda)A$. 
Then $A/J_{\lambda}$ is isomorphic to $S$. In particular, $J_{\lambda}$ is a prime ideal of $A$ for each $\lambda\in k$. 
\end{corollary}

\begin{proof} Let $\hat{\sigma}\in\mathcal{S}_a$ be as in {\it Thm.~\ref{unique}}. By {\it Thm.~\ref{A-generators}}, we have 
$A=k[h,\hat{\sigma}]$.

Given $f\in B$, let $f(0)$ denote the evaluation of $f$ at $a=0$. 
Since $D(a)=0$, 
we have $aB\cap A=aA$. 
Indeed, 
if $b\in B$ is such that $ab\in A$, 
then $aD(b)=D(ab)=0$, and so $D(b)=0$. 
Hence, the kernel of the map 
$A\to B$ defined by $f\to f(0)$ equals $aA$. Therefore:
\[
\mathfrak{A}:=A/aA\cong 
k[h(0),\sigma_0(0),\sigma _1(0),\sigma _2(0),\ldots ]
= k[h(0),x,xv,xv^2,\ldots ] 
=S[h(0)]= S^{[1]}
\]
The last equality holds because $h(0)=6x^3z-3x^2y^2$ is transcendental over $k[x,v]$. We conclude that:
\[
A/J_{\lambda} \cong 
\mathfrak{A}/(h(0)-\lambda )\mathfrak{A}\cong S
\]
\end{proof}

\subsection{Hypersurface Actions}\label{hypersurface} Given $\lambda\in k$, the $\G_a$-action on $\A^5$ defined by $D$ restricts to the hypersurface $X_{\lambda}\subset\A^5$ given by $h=\lambda$. We show that the ring of invariants for this action is a simple cable algebra of non-finite type.
Note that since $h$ does not involve $v$, $X_{\lambda}$ is a cylinder, i.e., $X_{\lambda}=Y_{\lambda}\times\A^1$ for the hypersurface $Y_{\lambda}\subset\A^4$ defined by $h=\lambda$. 

Let $\rho_{\lambda}: B\to B/(h-\lambda )B$ be the canonical surjection, and let $D_{\lambda}$ denote the locally nilpotent derivation of $B/(h-\lambda)B$ induced by $D$. 

\begin{theorem}\label{ker-Dbar} If $\hat{s}\in\mathcal{S}_a$, then 
$\krn D_{\lambda}=\rho_{\lambda}(A)=k[\rho_{\lambda}\hat{s}]$ 
and this ring is not finitely generated as a $k$-algebra. 
\end{theorem}

\begin{proof} Let $\lambda\in k$ be given, and set $K=A+(h-\lambda )B$. It must be shown that $D^{-1}((h-\lambda )B)=K$. The inclusion $K\subset D^{-1}((h-\lambda )B)$ is clear. For the converse, we first show that $K\cap aB=aK$.

Let $\hat{\sigma}\in\mathcal{S}_a$ be as in {\it Thm.~\ref{unique}}. By {\it Thm.~\ref{A-generators}}, we have 
$A=k[h,\hat{\sigma}]$. Therefore, $K=k[\hat{\sigma}]+(h-\lambda )B$. Given $g\in K$, write 
$g=p(\hat{\sigma })+(h-\lambda )b$ for $p\in\Omega$ and and $b\in B$. If $g\in aB$, then setting $a=0$ yields the following equation in 
$k[x,y,z,v]$:
\[
(3x^2(2xz-y^2)-\lambda )\cdot b|_{a=0}=((h-\lambda )b)|_{a=0} = -p(\hat{\sigma})|_{a=0}\in k[x,xv,xv^2,...]
\]
This means $b\in aB$, 
and hence $p(\hat{\sigma})=g-(h-\lambda )b\in aB$. 
Since $aB\cap A=aA$, 
it follows that $p(\hat{\sigma})\in aA$. 
Therefore, $g\in aK$. This shows $K\cap aB=aK$.
By induction on $n$, 
we get $K\cap a^nB=a^nK$ for each $n\ge 1$.

Now suppose that $f\in B$ and $Df\in (h-\lambda )B$. Since $v$ is a local slice of $D$ with $Dv=a^2$, there exists $n\ge 0$ and $P\in A^{[1]}$ such that $a^nf=P(v)=\sum_i\frac{1}{i!}P^{(i)}(0)v^i$. 
We thus have:
\begin{eqnarray*}
&P^{(i)}(v)a^{2i}& = D^iP(v)=a^nD^if\in (h-\lambda )B \quad \forall i\ge 1 \\
&\Rightarrow& P^{(i)}(v)\in (h-\lambda )B \quad \forall i\ge 1 \\
&\Rightarrow& P^{(i)}(0)\in (h-\lambda )R \quad \forall i\ge 1
\ (\text{since} \,\, B=R[v])
\\
&\Rightarrow& a^nf-P(0)=\sum_{i\ge 1}\frac{1}{i!}P^{(i)}(0)v^i
\in (h-\lambda )B
\end{eqnarray*}
Since $P(0)\in A$, 
we see that $a^nf$ belongs to $K$, 
and hence to $K\cap a^nB=a^nK$. 
Therefore, 
we get $f\in K$. We have thus shown:
\[
\krn D_{\lambda}=\rho_{\lambda}(A)=k[\rho_{\lambda}\hat{s}]
\]

It remains to show that $\rho_{\lambda}(A)$ is not finitely generated as a $k$-algebra. 
Since $D(h-\lambda )=0$, we have 
$\ker \rho_{\lambda}|_A=A\cap (h-\lambda )B
=(h-\lambda )A$. 
So $\rho_{\lambda}(A)\cong A/(h-\lambda )A$. 
By {\it Cor.~\ref{prime-ideal}}, it follows that:
\[
\rho_{\lambda}(A)/a\rho_{\lambda}(A) \cong A/J_{\lambda} \cong k[x,xv,xv^2,...]
\]
Since $\rho_{\lambda}(A)$ maps onto a non-finitely generated $k$-algebra, $\rho_{\lambda}(A)$ is not finitely generated as a 
$k$-algebra. 
\end{proof}


\section{Relations in $\bar{A}$}\label{A-bar}

We continue the notation of the preceding section. The main goal of this section is to show the following.

\begin{theorem}\label{Abar-relations} 
For every $\hat{s}\in\mathcal{S}_a$, 
we have $\ker \phi _{\pi \hat{s}}=\mathcal{Q}_4$. 
Consequently, $\bar{A}\cong_k\Omega/\mathcal{Q}_4$ 
by {\it Cor.~\ref{cor1}(a)}, 
and $\mathcal{Q}_4$ is a homogeneous prime ideal of $\Omega$.
\end{theorem} 

\subsection{Quadratic Relations} 
Let $\hat{s}\in\mathcal{S}_a$ be given, and let $\{\hat{\theta}_n\}$ be a $\Delta$-basis for $\Omega_2$, where $\hat{\theta}_n = (\theta_n^{(j)})$ for  given $n$.

\begin{lemma}\label{lemma2} 
\begin{itemize}
\item [{\bf (a)}] If $n\ge 4$ is even, then $\theta_n^{(j)}\in\krn\phi_{\pi\hat{s}}$ holds for any $j\ge 0$. 
\smallskip
\item [{\bf (b)}] $\langle \theta _0^{(j)},\theta _2^{(j-2)}\rangle \cap \krn\phi_{\pi\hat{s}}=\{ 0\} $ holds for every $j\ge 0$, 
where $\theta _2^{(j-2)}=0$ if $j=0,1$. 
\end{itemize}
\end{lemma}

\begin{proof} (a) Fixing $n\ge 4$, we proceed by induction on $j$ to show that $\theta_n^{(j)}\in\krn\phi_{\pi\hat{s}}$ for each $j\ge 0$. 
We have:
\[
\delta\phi_{\pi\hat{s}} (\theta_n^{(0)})=\phi_{\pi\hat{s}}\Delta (\theta_n^{(0)})=0 \quad\Rightarrow\quad \phi_{\pi\hat{s}} (\theta_n^{(0)})\in\krn\delta = k[\bar{a},\bar{F},\bar{G}]
\]
From line (\ref{FG}) in {\it Remark~\ref{list}}, 
we have that $\bar{F}=\phi_{\pi\hat{s}} (2x_0x_2-x_1^2)$ and $-\bar{G}=\phi_{\pi\hat{s}} (3x_0^2x_3-3x_0x_1x_2+x_1^3)$. 
Therefore, there exists $P\in\krn\phi_{\pi\hat{s}}\cap\Omega_{(2,n)}$ such that:
\[
\theta_n^{(0)}-P\in 
k[x_0,2x_0x_2-x_1^2,3x_0^2x_3-3x_0x_1x_2+x_1^3]\cap \Omega _2
=k\cdot x_0^2+k\cdot (2x_0x_2-x_1^2)
\subset \Omega_{(2,0)}+\Omega_{(2,2)}
\]
Since $\theta _n^{(0)},P\in \Omega _{(2,n)}$ and $n\ge 4$, 
we conclude $\theta_n^{(0)}=P\in\krn\phi_{\pi\hat{s}}$. 
This gives the basis for induction. 

Assume that $\theta_n^{(j-1)}\in\krn\phi_{\pi\hat{s}}$ for $j\ge 1$. Then:
\[
0=\phi_{\pi\hat{s}} (\theta_n^{(j-1)}) = \phi_{\pi\hat{s}}\Delta (\theta_n^{(j)}) = \delta\phi_{\pi\hat{s}} (\theta_n^{(j)}) \quad\Rightarrow\quad \phi_{\pi\hat{s}} (\theta_n^{(j)}) \in\krn\delta
\]
Since $\theta_n^{(j)}\in \Omega _{(2,n+j)}$, 
we conclude as above that 
$\theta_n^{(j)}\in\krn\phi_{\pi\hat{s}}$. 
This proves part (a).

(b) 
Since $\theta _0^{(j)}=x_0x_j\not\in\krn\phi_{\pi\hat{s}}$ 
for $j=0,1$, 
the assertion holds for $j=0,1$. 
By {\it Lem.~\ref{lemma1}(a)}, 
we have:
\[ 
\langle \phi _{\hat{s}}(\theta _0^{(2)}),
\phi _{\hat{s}}(\theta _2^{(0)})\rangle 
=\phi _{\hat{s}}(\Omega _{(2,2)})
=\phi _{\hat{s}}(\langle \beta _0^{(2)},\beta _2^{(0)}\rangle )
=\langle as_2,s_1^2\rangle 
\]
Since $\dim \langle as_2,s_1^2\rangle =2$ 
and 
$\langle as_2,s_1^2\rangle \cap hB
\subset B_{(6,2)}\cap hB=\{ 0\} $, 
the assertion also holds for $j=2$. 
We prove the case $j\ge 3$ by contradiction. 
Let $j\ge 3$ be the smallest integer 
for which there exists $(0,0)\ne (\alpha ,\beta )\in k^2$ 
such that 
$f:=\alpha \theta _0^{(j)}+\beta \theta _2^{(j-2)}
\in\krn\phi_{\pi\hat{s}}$. 
Then
\[
0=\phi_{\pi\hat{s}} (f) \quad\Rightarrow\quad 
0=\delta\phi_{\pi\hat{s}} (f) 
= \phi_{\pi\hat{s}}\Delta (f) 
= \phi_{\pi\hat{s}}
(\alpha \theta _0^{(j-1)}+\beta \theta _2^{(j-3)})
\]
and so 
$\alpha \theta _0^{(j-1)}+\beta \theta _2^{(j-3)}
\in\krn\phi_{\pi\hat{s}}$. 
This contradicts the minimality of $j$, 
proving part (b). 
\end{proof}

Combining {\it Lem.~\ref{lemma1}} and {\it Lem.~\ref{lemma2}}, we obtain:
\begin{lemma}\label{lemma3} 
\begin{itemize}
\item [{\bf (a)}] Given $j\ge 4$, the set $\{\theta_{2i}^{(j-2i)}\,\vert\, 2\le i\le j/2\}$ is a basis for $\Omega_{(2,j)}\cap\krn\phi_{\pi\hat{s}}$.
\medskip
\item [{\bf (b)}] The vertices of $\hat{\theta}_n$ $(n\in 2\N$, $n\ge 4)$ form a basis for $\Omega_2\cap\krn\phi_{\pi\hat{s}}$.
\end{itemize}
\end{lemma}

\subsection{Proof of Theorem~\ref{Abar-relations}} Note that, by {\it Cor.~\ref{dimension1} (a)}, if $\hat{t}\in \mathcal{S}_a$, then $\pi \hat{t}=\pi \hat{s}$. So there is no loss in generality in assuming that $\hat{s}=\hat{\sigma}$, where $\hat{\sigma}$ is the $\partial$-cable specified in {\it Thm.~\ref{unique}}.

By {\it Lem.~\ref{lemma3}(b)}, 
the ideal generated by $\Omega_2\cap\krn\phi_{\pi\hat{\sigma}}$ equals $\mathcal{Q}_4$. 
Since $\phi_{\pi\hat{\sigma}}$ 
is a homogeneous ideal of $\Omega $, 
it suffices to show: 
\[
\Omega_{(r,s)}\cap\krn\phi_{\pi\hat{\sigma}} \subset\mathcal{Q}_4 \quad (r,s\ge 0)
\]
Let non-zero $\zeta\in\Omega_{(r,s)}\cap\krn\phi_{\pi\hat{\sigma}}$ be given ($r,s\ge 0$). 
Then $r\ge 2$. 
We prove $\zeta \in \mathcal{Q}_4$ by induction on $r$, 
where the case $r=2$ holds as mentioned. 
Assume that $r\ge 3$. 
By {\it Thm.~\ref{Q-ideal}(c)} we have:
\[
\Omega_r\subset (x_0,x_1)^{r-1}+\mathcal{Q}_4
\]
So it suffices to assume that
$\zeta\in (x_0,x_1)^{r-1}$. By degree considerations, we see that $\zeta$ is a linear combination of the monomials:
\[
x_0^{r-i-1}x_1^ix_{s-i} \quad {\rm such \,\, that}\quad r-i-1, i, s-i \ge 0
\]
Suppose that $x_0$ does not divide $\zeta$. Then $s-r+1\ge 1$, and there exist $\zeta_0\in\Omega_{r-1}$ and non-zero $c\in k$ with $\zeta =x_0\zeta_0 + cx_1^{r-1}x_{s-r+1}$.
Since $\zeta\in\krn\phi_{\pi\hat{\sigma}}$, we see that $\phi_{\hat{\sigma}}(\zeta )\in hA$, which implies that, for some $q\in A$:
\begin{equation}\label{mod}
c\sigma_1^{r-1}\sigma_{s-r+1}=hq-a\phi_{\hat{\sigma}}(\zeta_0)
\end{equation}
By {\it Thm.~\ref{unique}}, we have that $n!\sigma_n\equiv -nxv^{n-1}\, ({\rm mod}\, aB)$ for each $n\ge 1$. From (\ref{mod}), it follows that:
\[
\frac{c}{(s-r)!}(-x)^rv^{s-r}=3x^2(2xz-y^2)\cdot q|_{a=0}
\]
Since $c\ne 0$, this is a contradiction. 
Therefore, $x_0$ divides $\zeta$. If $\zeta =x_0\zeta_0$ for $\zeta_0\in\Omega$, 
then $\zeta_0\in
\Omega _{(r-1,s)}\cap \krn\phi_{\pi\hat{\sigma}}$. 
We conclude by induction on $r$ that $\zeta_0\in\mathcal{Q}_4$. Therefore, $\zeta\in\mathcal{Q}_4$. 
This completes the proof of {\it Thm.~\ref{Abar-relations}}. $\Box$

\begin{example} {\rm 
Consider the well-known cubic $\Delta$-invariant given by:
\[
\xi = 2x_2^3+9x_0x_3^2-6x_1x_2x_3-12x_0x_2x_4+6x_1^2x_4 \\
\]
Let $\hat{\theta}_4$ be a $\Delta$-cable rooted at $\theta_4^{(0)}$ such that:
\[
\textstyle\theta_4^{(2)} = 5x_1x_5-8x_2x_4+\frac{9}{2}x_3^2\,\, ,\quad \theta_4^{(1)} = 5x_0x_5-3x_1x_4+x_2x_3 \,\, ,\quad \theta_4^{(0)}=2x_0x_4-2x_1x_3+x_2^2
\]
We have:
\[
\textstyle\frac{1}{2}\xi = x_0\theta^{(2)}_4-x_1\theta^{(1)}_4+x_2\theta^{(0)}_4 \in\mathcal{Q}_4
\]
Notice that, in order to express $\xi\in k[x_0,x_1,x_2,x_3,x_4]$ using quadratics in $\mathcal{Q}_4$, it was necessary to use $x_5$. }
\end{example}

\begin{example} {\rm Since the transcendence degree of $\bar{A}$ over $k$ is 3, $\bar{s_0},\bar{s}_1,\bar{s}_2,\bar{s}_3$ are algebraically dependent in $\bar{A}$. 
Their minimal algebraic relation is quartic and can be obtained as follows.

Let $\xi $ be as in the preceding example. 
The $x_4$-coefficient of $\xi$ is $-6\theta_2^{(0)}$, and the $x_4$-coefficient of $\theta_4^{(0)}$ is $2x_0$. Thus, in order to eliminate $x_4$, we take:
\[
\chi := 3\theta_2^{(0)}\theta_4^{(0)} + x_0\xi = 9x_0^2x_3^2-3x_1^2x_2^2+8x_0x_2^3-18x_0x_1x_2x_3+6x_1^3x_3
\]
We see that $\chi\in k[x_0,x_1,x_2,x_3]\cap\krn\Delta\cap\mathcal{Q}_4$. 
Since $\chi $ is irreducible, $\chi $ is a minimal algebraic relation among $\bar{s_0}$, $\bar{s}_1$, $\bar{s}_2$ and $\bar{s}_3$. }
\end{example}

\begin{remark}\label{3-term} {\rm Let $\hat{\eta}_4$ be the $\Delta$-cable belonging to the small $\Delta$-basis for $\Omega_2$. According to {\it Lem.~\ref{lemma2}}, 
$\hat{\eta_4}\subset\krn\phi_{\pi\hat{s}}$ for every $\hat{s}\in\mathcal{S}_a$. Recall that:
\[
\textstyle \eta_4^{(j)}=\frac{(j+1)(j+4)}{2}x_0x_{4+j}-(j+2)x_1x_{3+j}+x_2x_{2+j}
\]
Since we know $\bar{s}_0,\bar{s}_1,\bar{s}_2,\bar{s}_3$ (see {\it Remark~\ref{list}}), we can easily determine the $\delta$-cable $\pi\hat{s}$ using these 3-term recursion relations in $\bar{A}$. }
\end{remark}


\section{Relations in $A$} 

Let $\Omega [t]=\Omega^{[1]}$, and extend the $\Z^2$-grading on $\Omega$ to $\Omega [t]$ by setting $\deg t=(0,6)$. Note that
$\Omega [t]_r=\Omega_r[t]$ for each $r\ge 0$. In addition:
\[
\Omega [t]_{(r,n)} = \Omega_{(r,n)} \oplus t\cdot\Omega_{(r,n-6)}\oplus\cdots\oplus t^e\cdot\Omega_{(r,n-6e)} \,\,\, {\rm where} \,\,\, 0\le n-6e\le 5
\]
Extend $\Delta$ to $\tilde{\Delta}$ on $\Omega [t]$ by setting $\tilde{\Delta}(t)=0$. Then $\tilde{\Delta}$ is homogeneous and $\deg\tilde{\Delta}=(0,-1)$.
Since $\Delta :\Omega _{(r,s)}\to \Omega _{(r,s-1)}$ is surjective for each $r,s\ge 1$, we see that 
$\tilde{\Delta }:\Omega [t]_{(r,n)}\to \Omega [t]_{(r,n-1)}$ is surjective for each $r,n\ge 1$. 
Given $n\ge 0$, define the vector space:
\[
V_n=\Omega[t]_{(2,n)}\cap\krn\tilde{\Delta}
=\Omega [t]_{(2,n)}\cap(\krn\Delta )[t]
\]
Since $\krn\Delta\cap\Omega_{(2,s)}$ equals $\{0\}$ if $s$ is odd, and equals $k\cdot\theta_s^{(0)}$ if $s$ is even as mentioned 
in {\it Sect.~\ref{Delta-Basis}}, 
the reader can easily check that $V_n=\{ 0\}$ if $n$ is odd, and that for $n$ even:
\begin{equation}\label{eq:V_n n:even}
V_n=\langle \theta_n^{(0)},t\theta_{n-6}^{(0)},...,t^e\theta_{n-6e}^{(0)}\rangle \quad {\rm where}\quad n-6e\in\{ 0,2,4\}
\end{equation}
\subsection{The Mapping $\Phi_{\hat{s}}$}
By {\it Thm.~\ref{A-generators}(a)}, 
$\phi_{\hat{s}} :\Omega\to A$ extends to the surjection:
\[
\Phi_{\hat{s}} :\Omega [t]\to A \,\, ,\quad \Phi_{\hat{s}}(t) = h
\]
Note that $\Phi_{\hat{s}}\hat{\Delta }=\partial \Phi_{\hat{s}}$, 
since $\phi_{\hat{s}}\Delta =\partial \phi_{\hat{s}}$, 
$\Phi_{\hat{s}}\hat{\Delta }t=0$ 
and $\partial \Phi_{\hat{s}}t=0$. 

\begin{theorem}\label{A-relations} There exists a set $\{ \hat{\Theta}_4, \hat{\Theta}_6, \hat{\Theta}_8,...\}$ of homogeneous $\tilde{\Delta}$-cables such that $\hat{\Theta}_n$ is rooted in $V_n$ for each $n$ and:
\[
\krn\Phi_{\hat{s}} = (\hat{\Theta}_4, \hat{\Theta}_6, \hat{\Theta}_8,...)
\]
\end{theorem}

\begin{proof} The proof proceeds in three steps.

{\it Step 1.}  This step constructs a set $\{ \hat{\Theta}_4, \hat{\Theta}_6, \hat{\Theta}_8,...\}$ of homogeneous $\tilde{\Delta}$-cables such that $\hat{\Theta}_n$ is rooted in $V_n$ for each $n$ and
$(\hat{\Theta}_4, \hat{\Theta}_6, \hat{\Theta}_8,...)\subset\krn\Phi_{\hat{s}}$.
For the integer $n\ge 4$, write $n=6e+\ell$ $(e\ge 0, 0\le \ell\le 5)$. 
Given $P\in V_n$, we have:
\[
0=\Phi_{\hat{s}}\tilde{\Delta}(P) = \partial\Phi_{\hat{s}}(P) \quad\Rightarrow\quad \Phi_{\hat{s}}(V_n)\subset \ker \partial =R\cap A
\]
Since 
$\Phi_{\hat{s}}(V_n)\subset 
\Phi_{\hat{s}}(\Omega [t]_{(2,n)})\subset A_{(2n+2,n)}$, 
it follows that 
\begin{equation}\label{inclusion:9.5}
\Phi_{\hat{s}}(V_n)\subset R\cap A_{(2n+2,n)}=
\begin{cases} \langle a^2h^e\rangle & \ell =0\\ \langle Fh^e\rangle & \ell =2\\ \{ 0\} & {\rm otherwise}
\end{cases}
\end{equation}
by {\it Lem.~\ref{mini}(b)}. Now assume $n$ is even. 
In view of (\ref{eq:V_n n:even}), 
there exists $c_n\in k$ such that 
$\Phi_{\hat{s}}(\theta_n^{(0)})
=c_n\Phi_{\hat{s}}(t^e\theta_{\ell}^{(0)})$. 
Note that we may take $c_n=0$ when $\ell =4$. 
Then, we have:
\begin{equation}\label{root}
\Theta_n^{(0)}:=\theta_n^{(0)}-c_nt^e\theta_{\ell}^{(0)}\in\krn\Phi_{\hat{s}} -\{ 0\} ,
\end{equation}
since $e\ge 1$ except when $n=4$. 
Suppose that, for some $j\ge 0$, 
we have constructed 
$\Theta_n^{(0)},...,\Theta_n^{(j)}\in\krn\Phi_{\hat{s}}$ 
which satisfy $\Theta_n^{(i)}\in\Omega [t]_{(2,n+i)}$ and $\tilde{\Delta}\Theta_n^{(i)}=\Theta_n^{(i-1)}$, $1\le i\le j$. 
Since the mapping 
\[
\tilde{\Delta}:\Omega [t]_{(2,n+j+1)}\to\Omega [t]_{(2,n+j)}
\]
is surjective, we may choose $P\in\Omega [t]_{(2,n+j+1)}$ with $\tilde{\Delta}P=\Theta_n^{(j)}$. 
We have:
\[
0=\Phi_{\hat{s}}\Theta_n^{(j)}=\Phi_{\hat{s}}\tilde{\Delta}(P)=\partial\Phi_{\hat{s}}(P)
\quad\Rightarrow\quad 
\Phi_{\hat{s}}(P)\in R\cap A_{(2(n+j+1)+2,n+j+1)}
\]
We again apply the equality in (\ref{inclusion:9.5}). 
In fact, 
if $\ell\in \{ 0,2\} $, 
then $\theta _{n-6e}^{(0)}
=\theta _{\ell }^{(0)}\not\in \ker \phi _{\pi \hat{s}}$ 
by {\it Lem.~\ref{lemma2}(b)}, 
and so $\Phi_{\hat{s}}(t^e\theta _{n-6e}^{(0)})
=h^e\phi _{\hat{s}}(\theta _{n-6e}^{(0)})\ne 0$. 
Thus, as above, 
there exist $\kappa\in k$ and $\epsilon ,l\in\N$ with:
\[
\Theta_n^{(j+1)}:=P-\kappa t^{\epsilon}\theta_l^{(0)} \in\krn\Phi_{\hat{s}}\cap \Omega [t]_{(2,n+j+1)}
\]
where $\kappa =0$ if $n+j+1$ is odd, since $V_{n+j+1}=\{ 0\} $. 
Then, we have $\tilde{\Delta}\Theta_n^{(j+1)}=\Theta_n^{(j)}$, 
since $\tilde{\Delta }(t^{\epsilon}\theta_l^{(0)})=0$. 
Therefore, for each even $n\ge 4$, there exists a homogeneous $\tilde{\Delta}$-cable $\hat{\Theta}_n$ rooted in $V_n$ 
and contained in $\krn\Phi_{\hat{s}}\cap\Omega [t]_2$. 
Note that $\Theta_4^{(j)}=\theta_4^{(j)}$ for $j=0,1$ 
by construction.

{\it Step 2.} By construction, the ideal 
$J:=(\hat{\Theta}_4, \hat{\Theta}_6, \hat{\Theta}_8,...)$ of $\Omega [t]$ is contained in $\krn\Phi_{\hat{s}}$. 
This step shows $\krn\Phi_{\hat{s}}\subset J+(t)$. Define polynomials $H_n^{(j)}\in\Omega_{(2,n+j)}$ ($n\in 2\N$, $n\ge 4$, $j\ge 0$) by $H_n^{(j)}=\Theta_n^{(j)}\vert_{t=0}$. 
Note that, by (\ref{root}), 
we have $H_n^{(0)}=\theta_n^{(0)}\ne 0$. 
Therefore, by {\it Sect.~\ref{D-cables} (xi)}, for each even $n\ge 4$, $\hat{H}_n:=(H_n^{(j)})$ is a homogeneous $\Delta$-cable rooted at $\theta_n^{(0)}$.
By {\it Def.~\ref{def:Q-ideal}(2)} 
and {\it Lem.~\ref{Q-ideal}(b)}, we get 
$$
\mathcal{Q}_4+(t)=
(\hat{H}_4,\hat{H}_6,\hat{H}_8,...)+(t) = 
(\hat{\Theta}_4, \hat{\Theta}_6, \hat{\Theta}_8,...)+(t)
=J+(t)
$$
Consider the map 
$\pi \Phi_{\hat{s}} :\Omega [t]
\stackrel{\Phi_{\hat{s}}}{\to }A\stackrel{\pi }{\to }A/hA$. 
Since $\pi \Phi_{\hat{s}}|_{\Omega }=\phi_{\pi\hat{s}}$, 
we see from {\it Thm.~\ref{Abar-relations}} that:
\[
\krn\Phi_{\hat{s}}\subset\krn\pi\Phi_{\hat{s}}
=\mathcal{Q}_4+(t)=J+(t)
\]

{\it Step 3.} This step shows $J=\krn\Phi_{\hat{s}}$. 
Since $\Phi_{\hat{s}}(\Omega [t]_{(r,s)})\subset A_{(2s+r,s)}$ 
for each $r,s\ge 0$, we see that $\ker \Phi_{\hat{s}}$ 
is a homogeneous ideal of $\Omega [t]$. 
So, given integers $r,N\ge 0$, we show by induction on $N$ that:
\begin{equation}\label{induct-N}
\krn\Phi_{\hat{s}}\cap\Omega[t]_{(r,N)}\subset J
\end{equation}
If $r\le 1$, then $\krn\Phi_{\hat{s}}\cap\Omega[t]_{(r,N)}=\{ 0\}$, so assume $r\ge 2$.

Consider first the case that $0\le N\le 5$. 
In this case, 
$\Omega [t]_{(r,N)}=\Omega _{(r,N)}=k[x_0,\ldots ,x_N]_{(r,N)}$, 
since $\deg t=(0,6)$. 
Let 
\[
P\in \krn\Phi_{\hat{s}}\cap\Omega[t]_{(r,N)} = \krn\phi_{\hat{s}}\cap k[x_0,...,x_N]_{(r,N)}
\]
be given. 
If $N\le 3$ then $P=0$, since $s_0,s_1,s_2,s_3$ are algebraically independent over $k$ ({\it Remark~\ref{list}}). 

Suppose that $N=4$. 
The only monomial in $k[x_0,...,x_4]_{(r,4)}$ in which $x_4$ appears is $x_0^{r-1}x_4$. Therefore, 
noting $\theta _4^{(0)}=2(x_0x_4-x_1x_3)+x_2^2$, 
we have: 
\[
k[x_0,...,x_4]_{(r,4)}=k\cdot x_0^{r-1}x_4\oplus k[x_0,...,x_3]_{(r,4)} = k\cdot x_0^{r-2}\theta_4^{(0)}\oplus k[x_0,...,x_3]_{(r,4)}
\]
So there exists $\lambda\in k$ such that 
$P-\lambda x_0^{r-2}\theta_4^{(0)}\in k[x_0,...,x_3]$. 
Since $\theta_4^{(0)}\in \ker \phi_{\hat{s}}$ 
by {\it Lem.~\ref{universal}(a)} below, 
we get $P-\lambda x_0^{r-2}\theta_4^{(0)}\in \krn\phi_{\hat{s}}\cap k[x_0,...,x_3]=\{ 0\}$. 
Since $\theta_4^{(0)}=\Theta_4^{(0)}\in J$, $P\in J$ in this case.

Suppose that $N=5$. 
The only monomial in $k[x_0,...,x_5]_{(r,5)}$ in which $x_5$ appears is $x_0^{r-1}x_5$. Therefore, 
noting $\theta _4^{(1)}=5x_0x_5-3x_1x_4+x_2x_3$, 
we have: 
\[
k[x_0,...,x_5]_{(r,5)}=k\cdot x_0^{r-1}x_5\oplus k[x_0,...,x_4]_{(r,5)} = k\cdot x_0^{r-2}\theta_4^{(1)}\oplus k[x_0,...,x_4]_{(r,5)}
\]
Since $\theta_4^{(1)}\in \ker \phi_{\hat{s}}$ by {\it Lem.~\ref{universal}(a)} below, 
there exists $\lambda\in k$ such that 
\[
P-\lambda x_0^{r-2}\theta_4^{(1)}\in \krn\phi_{\hat{s}}\cap k[x_0,...,x_4]_{(r,5)}
\]
as above. Similarly, the only monomial in $k[x_0,...,x_4]_{(r+1,5)}$ in which $x_4$ appears is $x_0^{r-1}x_1x_4$. Therefore:
\[
k[x_0,...,x_4]_{(r+1,5)}=k\cdot x_0^{r-1}x_1x_4\oplus k[x_0,...,x_3]_{(r+1,5)} = k\cdot x_0^{r-2}x_1\theta_4^{(0)}\oplus k[x_0,...,x_3]_{(r+1,5)}
\]
So there exists $\mu\in k$ such that:
\[
x_0P-\lambda x_0^{r-1}\theta_4^{(1)}-\mu x_0^{r-2}x_1\theta_4^{(0)}\in \krn\phi_{\hat{s}}\cap k[x_0,...,x_3]=\{ 0\}
\]
If $r=2$, then $\mu x_1\theta_4^{(0)}\in x_0\Omega$ implies $\mu =0$ and $P=\lambda x_0^{r-2}\theta_4^{(1)}$. If $r\ge 3$, then:
\[
P=\lambda x_0^{r-2}\theta_4^{(1)}+\mu x_0^{r-3}x_1\theta_4^{(0)}
\]
In either case, $P\in J$, since $\theta_4^{(1)}=\Theta_4^{(1)}\in J$.
Therefore, the inclusion (\ref{induct-N}) holds when $0\le N\le 5$, which gives the basis for induction. 

Suppose that $N_0$ is an integer such that $N_0\ge 5$ and (\ref{induct-N}) holds for all integers $0\le N\le N_0$. 
Let $P\in\krn\Phi_{\hat{s}}\cap\Omega[t]_{(r,M)}$ be given, 
where $N_0<M\le N_0+6$. 
We show that $P$ is of the form:
\begin{equation}\label{peejay}
 P=P_J+tQ \quad {\rm where}\quad P_J\in J\cap\Omega [t]_{(r,M)} \quad {\rm and}\quad Q\in\Omega[t]_{(r,M-6)}
 \end{equation}
Since $\krn\Phi_{\hat{s}}\subset J+(t)$ by Step 2, 
we may write $P=E+C$ for $E\in J$ and $C\in t\cdot \Omega [t]$. Since $J$ and $t\cdot\Omega [t]$ are homogeneous ideals, each homogeneous summand of $E$ belongs to $J$, and each homogeneous summand of $C$ belongs to $t\cdot\Omega [t]$. Since $P$ is homogeneous, statement (\ref{peejay}) holds.

In addition, since $P_J\in J\subset \ker \Phi_{\hat{s}}$, 
we have:
\[ tQ=P-P_J\in\krn\Phi_{\hat{s}} \quad\Rightarrow\quad 
Q\in\krn\Phi_{\hat{s}}\cap\Omega [t]_{(r,M-6)}
\]
By the inductive hypothesis, $Q\in J$, which implies $P\in J$. Therefore, statement (\ref{induct-N}) holds for all $N\ge 0$. 
This proves $J=\krn\Phi_{\hat{s}}$.
\end{proof}

\subsection{The Kernel of $\phi_{\hat{\sigma}}$} 
The preceding section shows the existence of certain $\tilde{\Delta}$-cables $\hat{\Theta}_n\in\krn\Phi_{\hat{s}}$, but it is unclear how to construct these. In particular, we do not know the constants $c_n\in k$ in line (\ref{root}) when $n=6e$ or $n=6e+2$, although we do know $c_n=0$ when $n=6e+4$. These constants depend on the choice of the cable $\hat{s}$, and if $\hat{s}$ is known, then the cables $\hat{\Theta}_n$ can be constructed explicitly from $\hat{s}$. 
We will show that $\hat{s}$ can be chosen so that $c_n=0$ when $n=6e+2$. This gives a set of kernel elements large enough to define the sequence $s_n$ implicitly. 

Note that one could also find the sequence $s_n$ explicitly: first, calculate the sequence $w_n$ in the proof of {\it Thm.~\ref{cable}} by methods of linear algebra, and then calculate $s_n$ using the Dixmier map $\pi_v$. But the implicit method is clearly more efficient once relations have been established. 

Let $\hat{\sigma}\in \mathcal{S}_a$ be the $\partial$-cable defined in {\it Thm.~\ref{unique}}. 

\begin{lemma}\label{universal} Let $n\in 2\N$, $n\ge 4$, and $\hat{s}\in\mathcal{S}_a$ be given. 
\begin{itemize}
\item [{\bf (a)}] If $n\equiv 4\, ({\rm mod}\, 6)$, then $\theta_n^{(0)},\theta_n^{(1)}\in\krn\phi_{\hat{s}}$ for every $\hat{s}\in\mathcal{S}_a$. 
\medskip
\item [{\bf (b)}] If $n\equiv 2\, ({\rm mod}\, 6)$, then $\theta_n^{(0)},\theta_n^{(1)}\in\krn\phi_{\hat{\sigma}}$. 
\end{itemize}
\end{lemma}

\begin{proof} For both (a) and (b), 
it suffices to show that $\theta_n^{(0)}\in\krn\phi_{\hat{s}}$, 
since 
\[
0=\phi_{\hat{s}} (\theta_n^{(0)}) = \phi_{\hat{s}}\Delta (\theta_n^{(1)}) = \partial\phi_{\hat{s}}(\theta_n^{(1)}) \quad\Rightarrow\quad 
\phi_{\hat{s}}(\theta_n^{(1)} )\in\krn 
\partial |_{A_{(2n+4,n+1)}}
=R\cap A_{(2n+4,n+1)}
=\{ 0\}
\]
by {\it Lem.~\ref{mini}(b)} with $\ell =5,3$. 
If $n\equiv 4\, ({\rm mod}\, 6)$, then inclusion (\ref{inclusion:9.5}) shows that $\theta_n^{(0)}\in\krn\phi_{\hat{s}}$.
This proves part (a). 
For part (b), write $n=6e+2$ for some $e\ge 1$. 
Inclusion (\ref{inclusion:9.5}) shows that:
 \[
 \phi_{\hat{\sigma}}(\theta_n^{(0)})=cFh^e \quad (c\in k)
 \]
By {\it Thm.~\ref{unique}(b)}, it follows that $\phi_{\hat{\sigma}}(\theta_n^{(0)})=0$. This proves part (b).
\end{proof}

For $n\in 2\N$, let $J_n$ be the set of integers $j\ge 3$ such that $n+j\equiv 1\, (\text{mod}\, 6)$. In particular, each $j\in J_n$ is odd. 

Let $\{\hat{\theta}_n\}$  be a $\Delta$-basis for $\Omega_2$. Given $n\in 2\N$ and $j\in\N$ (and $j\ge 1$ if $n=0$), let $\xi(\theta_n^{(j)})\in k$ be the coefficient of $x_1x_{n+j-1}$ in $\theta_n^{(j)}$.
Note that $\xi(\theta_n^{(j)})=0$ if and only if $\theta_n^{(j)}\in k[x_0,x_2,x_3,\ldots ,x_{n+j}]$, since $\theta_n^{(j)}\in\Omega _{(2,n+j)}$. 
Define
\[
\mu (\hat{\theta}_n)=\min\{ j\in J_n\, |\, \xi(\theta_n^{(j)})\ne 0\}
\]
where it is understood that $\mu (\hat{\theta}_n)=\infty$ if $\xi(\theta_n^{(j)})=0$ for all $j\in J_n$. 

\begin{lemma}\label{(P)-constrained equivalence} If $\mu (\hat{\theta}_n)=\infty$, then the following are equivalent.
\begin{itemize}
\smallskip
\item  [{\rm (i)}] $\theta_n^{(j)}\in\krn\phi_{\hat{\sigma}}$ for some $j\ge 0$
\smallskip
\item [{\rm (ii)}] $\theta_n^{(0)}\in\krn\phi_{\hat{\sigma}}$
\smallskip
\item [{\rm (iii)}] $\theta_n^{(j)}\in\krn\phi_{\hat{\sigma}}$ for all $j\ge 0$
\end{itemize}
\end{lemma}
\begin{proof}
It is clear that (i) $\Leftarrow$ (ii) $\Leftarrow$ (iii). 
We also have (i) $\Rightarrow $ (ii), 
since 
$$
\phi_{\hat{\sigma}}(\theta _n^{(0)})
=\phi_{\hat{\sigma}}(\Delta ^j\theta _n^{(j)})
=\partial ^j\phi_{\hat{\sigma}}(\theta _n^{(j)})=0
$$
We show (ii) $\Rightarrow $ (iii). 
Suppose that $\theta_n^{(0)}\in\krn\phi_{\hat{\sigma}}$, 
noting that $n\ge 4$, 
since $\phi_{\hat{\sigma}}(\theta_0^{(0)})$ and $\phi_{\hat{\sigma}}(\theta_2^{(0)})$ cannot be zero by {\it Lem.~\ref{lemma2}(b)}. 
We prove by induction on $j$ that $\theta_n^{(j)}\in\krn\phi_{\hat{\sigma}}$ for all $j\ge 0$. 

Assume that 
$\theta_n^{(j)}\in\krn\phi_{\hat{\sigma}}$ for some $j\ge 0$. 
Then, 
$\partial \phi_{\hat{\sigma}}(\theta_n^{(j+1)})
=\phi_{\hat{\sigma}}(\Delta \theta_n^{(j+1)})
=\phi_{\hat{\sigma}}(\theta_n^{(j)})=0$. 
Hence, we get 
\[
\phi_{\hat{\sigma}}(\theta_n^{(j+1)})\in 
\krn \partial |_{A_{(2(n+j+1)+2,n+j+1)}}
=R\cap A_{(2(n+j+1)+2,n+j+1)}
\]
Now, suppose that 
$\theta_n^{(j+1)}\not\in\krn\phi_{\hat{\sigma}}$. 
Then, by {\it Lem.~\ref{mini}(b)} 
and the remark after {\it Thm.~\ref{unique}}, 
we have $n+j+1\equiv 0\pmod{6}$ 
and $\phi_{\hat{\sigma}}(\theta_n^{(j+1)})=\lambda a^2h^e$ 
for some $\lambda \in k^*$ and $e\ge 0$. 
Note that 
\[
\partial :{A_{(2(n+j+2)+2,n+j+2)}}\to{A_{(2(n+j+1)+2,n+j+1)}}
\]
is an injection by {\it Lem.~\ref{mini}(b)}, 
since $n+j+2\equiv 1\pmod{6}$. 
Because 
$\partial \phi_{\hat{\sigma}}(\theta_n^{(j+2)})
=\phi_{\hat{\sigma}}(\Delta \theta_n^{(j+2)})
=\phi_{\hat{\sigma}}(\theta_n^{(j+1)})$ and 
$\partial \sigma _0\sigma _1h^e=a^2h^e$, 
it follows that 
$\phi_{\hat{\sigma}}(\theta_n^{(j+2)})=\lambda \sigma _0\sigma _1h^e$. 
By assumption, 
the monomial $x_1x_{n+j+1}$ does not appear in $\theta _n^{(j+2)}$. 
Hence, 
$\theta _n^{(j+2)}$ is a $k$-linear combination 
of $x_ix_{n+j+2-i}$ for $0\le i\le (n+j+2)/2$ with $i\ne 1$. 
This contradicts {\it Thm.~\ref{unique}(a)}. 
Therefore, 
we must have $\theta_n^{(j+1)}\in\krn\phi_{\hat{\sigma}}$. 
It follows by induction that $\theta_n^{(j)}\in\krn\phi_{\hat{\sigma}}$ for all $j\ge 0$. This completes the proof. 
\end{proof}

Combining  {\it Lem.~\ref{universal}} and {\it Lem.~\ref{(P)-constrained equivalence}} gives the following result. 

\begin{lemma}\label{contain} Suppose that $\{\hat{\theta}_n\}$ is a $\Delta$-basis such that $\mu (\hat{\theta}_n)=\infty$ for each $n=6e\pm 2$, $e\ge 1$. Define the $Q$-ideal 
$\mathcal{J}$ by $\mathcal{J}=(\hat{\theta}_n\, |\, n=6e\pm 2, e\ge 1 )$. Then $\mathcal{J}\subset \krn\phi_{\hat{\sigma}}$. 
\end{lemma}

We next describe a procedure to modify a given $\Delta$-basis $\{\hat{\theta}_n\}$ to obtain a $\Delta$-basis $\{\hat{\psi}_n\}$ for which $\mu (\hat{\psi}_n)=\infty$ for each $n$. 

Given $n\in 2\N$, if $\mu (\hat{\theta_n})=\infty$, set $\hat{\psi}_n=\hat{\theta}_n$.  If $\mu (\hat{\theta}_n)<\infty$, then 
define constants
\[
j=\mu (\hat{\theta}_n)\,\, ,\,\, m=j-1 \quad\text{and}\quad  c=\frac{\xi(\theta_n^{(j)})}{n+j-2}
\]
noting that $j\ge 3$ is odd and $\xi(\theta_{n+j-1}^{(1)})=-(n+j-2)\ne 0$. It follows that:
\[
\mu (\hat{\theta}_n) < \mu  (\hat{\theta}_n+_mc\,\hat{\theta}_{n+m})
\]
If $\mu  (\hat{\theta}_n+_mc\,\hat{\theta}_{n+m})=\infty$, set $\hat{\psi}_n=\hat{\theta}_n+_mc\,\hat{\theta}_{n+m}$. 
If $\mu  (\hat{\theta}_n+_mc\,\hat{\theta}_{n+m})<\infty$, the process can be repeated. Continuing in this way, we construct a strictly increasing sequence $\vec{m}=\{ m_i\}_{i\in I}$ of positive integers, together with sequences $\vec{c}=\{ c_i\}_{i\in I}$ for $c_i\in k^*$ 
and $\vec{s}=\{ \hat{\theta}_{n+m_i}\}_{i\in I}$ such that, if
$\hat{\psi}_n=\lim (\vec{s},\vec{m},\vec{c})$, 
then $\mu (\hat{\psi}_n)=\infty$. 

Note that, with this algorithm, $\{ \hat{\psi}_n\} $ is uniquely determined by $\{ \hat{\theta}_n\} $. 
The resulting $\Delta$-basis $\{\hat{\psi}_n\}$ is the {\it reduction} of $\{\hat{\theta}_n\}$. 

To illustrate, let $\{\hat{\psi}_n\}$ be the reduction of the balanced $\Delta $-basis $\{\hat{\beta}_n\}$.
Assume that $n\equiv 4\pmod{6}$. 
Then $\xi(\beta _n^{(3)})=-\binom{n+2}{3}$, and if
\[
c=-\frac{\binom{n+2}{3}}{n+3-2}=-\frac{n(n+2)}{6}
\]
then the first eight terms of $\hat{\psi}_n$ equal those of $\hat{\beta}_n+_2c\,\hat{\beta}_{n+2}$. 
In particular, we have: 
\begin{equation}\label{psi2}
\psi _n^{(2)} =\beta _n^{(2)}-\dfrac{n(n+2)}{6}\,\beta _{n+2}^{(0)}=\frac{1}{6}\sum_{i=0}^{n+2}(-1)^i\left( 3i(i-1)-n(n+2)\right) x_ix_{n+2-i}\\
\end{equation}
and
\begin{equation}\label{psi3}
\psi _n^{(3)} =\beta _n^{(3)}-\dfrac{n(n+2)}{6}\,\beta _{n+2}^{(1)}=\frac{1}{6}\sum_{i=1}^{n+3}(-1)^{i+1}\left( (i-1)(i-2)-n(n+2)\right) i x_ix_{n+3-i}
\end{equation}
Note that, by {\it Lem.~\ref{contain}}, $\psi _n^{(2)}$ and $\psi _n^{(3)}$ above both belong to $\krn\phi_{\hat{\sigma}}$.

\begin{remark} {\rm The results of this section show that a $\Delta$-basis of the type described in {\it Lem.~\ref{contain}} exists, and therefore 
$\mathcal{J} \subset \krn\phi_{\hat{\sigma}}$ for the associated $Q$-ideal $\mathcal{J}$.
But we do not know if $\mathcal{J} = \krn\phi_{\hat{\sigma}}$.}
\end{remark}

\subsection{The Cable $\hat{\sigma}$} Let $\{\hat{\psi}_n\}$ be the reduction of the balanced $\Delta $-basis $\{\hat{\beta}_n\}$.
The $\Delta$-cables $\hat{\psi}_n$ for $n=6e\pm 2$ $(e\ge 1)$ give us a way to implicitly calculate the $\partial$-cable $\hat{\sigma}$. 
Recall that $\sigma_0,...,\sigma_5$ are uniquely determined and are given in {\it Remark~\ref{list}}. 

\begin{theorem}\label{A-constructs} 
For $n\ge 2$, we have: 
\begin{align*}
\sigma _n&
=-\frac{1}{2a}\sum _{i=1}^{n-1}(-1)^i\sigma_i\sigma_{n-i}
&\quad \text{if}\quad n\equiv 2,4\pmod{6}\\
\sigma _n&
=-\dfrac{1}{na}\sum _{i=1}^{n-1}(-1)^{i+1}i\sigma_i\sigma_{n-i}
&\quad \text{if}\quad n\equiv 3,5\pmod{6}\\
\sigma _n&=-\dfrac{1}{n(n+1)a}\sum _{i=1}^{n-1}
(-1)^i\left( 3i(i-1)-n(n-2)\right) \sigma_i\sigma_{n-i}
&\quad \text{if}\quad n\equiv 0\pmod{6}\\
\sigma _n&=-\dfrac{1}{n(n-1)a}\sum _{i=1}^{n-1}
(-1)^{i+1}\left( (i-1)(i-2)-(n-1)(n-3)\right)i\sigma_i\sigma_{n-i}
&\quad \text{if}\quad n\equiv 1\pmod{6}
\end{align*}
\end{theorem}
\begin{proof}
The first two equalities are equivalent to 
$\phi _{\hat{\sigma }}(\theta _n^{(0)})=0$ and 
$\phi _{\hat{\sigma }}(\theta _{n-1}^{(1)})=0$, 
respectively, 
which follow from {\it Lem.~\ref{universal}}. 
The last two equalities follow from {\it Lem.~\ref{contain}} together with (\ref{psi2}) and (\ref{psi3}).
\end{proof}

To illustrate, the following relations can be used to construct $\sigma_6,...,\sigma_{19}$. 
\medskip
{\small
\begin{eqnarray*}
\psi_4^{(2)} = \beta_4^{(2)}-4\beta_6^{(0)} &=& 7x_0x_6-2x_1x_5-x_2x_4+x_3^2 \\
\psi_4^{(3)} = \beta_4^{(3)}-4\beta_6^{(1)} &=& 7x_0x_7-2x_2x_5+x_3x_4 \\
\psi_8^{(0)} = \beta_8^{(0)} &=& 2x_0x_8-2x_1x_7+2x_2x_6-2x_3x_5+x_4^2 \\
\psi_8^{(1)} = \beta_8^{(1)} &=& 9x_0x_9-7x_1x_8+5x_2x_7-3x_3x_6+x_4x_5 \\
\psi_{10}^{(0)} = \beta_{10}^{(0)} &=& 2x_0x_{10}-2x_1x_9+2x_2x_8-2x_3x_7+2x_4x_6-x_5^2 \\
\psi_{10}^{(1)} = \beta_{10}^{(1)} &=& 11x_0x_{11}-9x_1x_{10}+7x_2x_9-5x_3x_8+3x_4x_7-x_5x_6 \\
\psi_{10}^{(2)} = \beta_{10}^{(2)}-20\beta_{12}^{(0)} &=& 26x_0x_{12}-15x_1x_{11}+6x_2x_{10}+x_3x_9-6x_4x_8+9x_5x_7-5x_6^2 \\
\psi_{10}^{(3)} = \beta_{10}^{(3)}-20\beta_{12}^{(1)} &=& 26x_0x_{13}-15x_2x_{11}+21x_3x_{10}-20x_4x_9+14x_5x_8-5x_6x_7\\
\psi_{14}^{(0)} = \beta_{14}^{(0)} &=& 2x_0x_{14}-2x_1x_{13}+2x_2x_{12}-2x_3x_{11}+2x_4x_{10}-2x_5x_9+2x_6x_8-x_7^2 \\
\psi_{14}^{(1)} = \beta_{14}^{(1)} &=& 15x_0x_{15}-13x_1x_{14}+11x_2x_{13}-9x_3x_{12}+7x_4x_{11}-5x_5x_{10}+3x_6x_9-x_7x_8 \\
\psi_{16}^{(0)} = \beta_{16}^{(0)} &=& 2x_0x_{16}-2x_1x_{15}+2x_2x_{14}-2x_3x_{13}+2x_4x_{12}-2x_5x_{11}+2x_6x_{10}-2x_7x_9+x_8^2 \\
\psi_{16}^{(1)} = \beta_{16}^{(1)} &=& 17x_0x_{17}-15x_1x_{16}+13x_2x_{15}-11x_3x_{14}+9x_4x_{13}-7x_5x_{12}+5x_6x_{11}-3x_7x_{10} \\
                                                            & &  +x_8x_9 \\
\psi_{16}^{(2)} = \beta_{16}^{(2)}-48\beta_{12}^{(0)} &=& 57x_0x_{18}-40x_1x_{17}+25x_2x_{16}-12x_3x_{15}+x_4x_{14}+8x_5x_{13}-25x_6x_{12}+20x_7x_{11} \\
                                                            & &  -23x_8x_{10}+12x_9^2 \\
\psi_{16}^{(3)} = \beta_{16}^{(3)}-48\beta_{12}^{(1)} &=& 57x_0x_{19}-40x_2x_{17}+65x_3x_{16}-77x_4x_{15}+78x_5x_{14}-70x_6x_{13}+55x_7x_{12} \\
                                                            & &  -35x_8x_{11}+12x_9x_{10}
\end{eqnarray*}
}
\begin{remark} {\rm The reader can compare these relations to relations for the sequence $w_n$ given in \cite{Freudenburg.06}. In particular, $w_0,...,w_{13}$ are given on p.162 and p.165. The construction used there is as follows: Given $n=6e-4$ $(e\ge 2)$, 
suppose $w_0,...,w_{6e-5}$ are known. Then $w_{6e-4},...,w_{6e+1}$ are defined by solving certain systems of linear equations, but in the language of cables this amounts to finding 
$\psi_n^{(0)},...,\psi_n^{(5)}$. Our current approach uses the simpler relations:
\[
\psi_n^{(0)}\,\, , \,\, \psi_n^{(1)}\,\, ,\,\, \psi_{n+2}^{(0)}\,\, ,\,\, \psi_{n+2}^{(1)}\,\, ,\,\, \psi_{n+2}^{(2)}\,\, ,\,\, \psi_{n+2}^{(3)}
\]
Note that, if $\psi_n^{(j)}=\sum_{i=0}^nc_{(n,i)}^{(j)}x_ix_{n-i}$, then the coefficient $c_{(n,i)}^{(j)}$ is a polynomial of degree $j$ in $i$. 
Thus, using smaller $j$-values has a big advantage computationally.
However, the reader should note that both methods produce the same sequence $\sigma_n$, by the uniqueness established in {\it Thm.~\ref{unique}}. }
\end{remark}


\section{Roberts' Derivations in Dimension Seven}

In \cite{Roberts.90}, Roberts constructed a family of counterexamples to Hilbert's Fourteenth Problem in the form of subrings $\mathcal{A}_m\subset k^{[7]}$ for integers $m\ge 2$. Although Roberts does not use the language of derivations, the maps he defines are triangular derivations. In this section, we give a description of the ring $\mathcal{A}_2$ as a cable algebra. 

Let $\mathcal{B}=k[X,Y,Z,S,T,U,V]=k^{[7]}$. For $m\ge 2$, the subring 
$\mathcal{A}_m$ is the kernel of the derivation $\mathcal{D}_m$ of $\mathcal{B}$ defined by:
\[
S\to X^{m+1}\,\, ,\,\, T\to Y^{m+1}\,\, ,\,\, U\to Z^{m+1}\,\, ,\,\, V\to (XYZ)^m\,\, ,\,\, X,Y,Z\to 0
\]
Define $H_m\in\mathcal{A}_m$ by $H_m=Y^{m+1}S-X^{m+1}T$. Define an action of the cyclic group $\Z_3=\langle\alpha\rangle$ on $\mathcal{B}$ by:
\[
\alpha (X,Y,Z,S,T,U,V) = (Z,X,Y,U,S,T,V)
\]
Then $\alpha$, $\mathcal{D}_m$ and the partial derivative $\partial/\partial V$ commute pairwise with each other. Therefore, $\alpha$ and $\partial/\partial V$ restrict to $\mathcal{A}_m$. 
We denote the restriction of $\partial/\partial V$ to $\mathcal{A}_m$ by $\delta_m$. 

Let $m\ge 2$ be given. In Lemma~3 of \cite{Roberts.90}, Roberts showed the existence of a sequence in $\mathcal{A}_m$ of the form $XV^i+(\text{terms of lower degree in $V$})$, $i\ge 0$. 
By combining this with homogeneity conditions, he concluded that $\mathcal{A}_m$ is not finitely generated over $k$. Note that, by applying $\alpha$, we also obtain sequences in $\mathcal{A}_m$ of the 
form $YV^i+(\text{terms of lower degree in $V$})$ and $ZV^i+(\text{terms of lower degree in $V$})$ for $i\ge 0$. 
The second author showed the following.

\begin{theorem}[Thm.~3.3 of \cite{Kuroda.04b}]\label{Kuroda04b} 
Given $m\ge 2$, let $I_{(m,X,i)}, I_{(m,Y,i)}, I_{(m,Z,i)}\in\mathcal{A}_m$ $(i\ge 0)$ be sequences of the form:
\begin{eqnarray*}
I_{(m,X,i)}=XV^i+(\text{terms of lower degree in $V$}) \\
I_{(m,Y,i)}=YV^i+(\text{terms of lower degree in $V$}) \\
I_{(m,Z,i)}=ZV^i+(\text{terms of lower degree in $V$})
\end{eqnarray*}
Then:
\[
\mathcal{A}_m=k[\{ H_m,\alpha H_m,\alpha^2 H_m\} \cup \{ I_{(m,W,i)}\mid i\ge 0,W\in \{ X,Y,Z\}\} ]
\]
\end{theorem}

We use this to show:

\begin{theorem}\label{Roberts} There exists an infinite $\delta_2$-cable $\hat{P}$ in $\mathcal{A}_2$ rooted at $X$, and for any such $\hat{P}$ we have:
\[
\mathcal{A}_2 = k[H_2,\alpha H_2,\alpha^2 H_2, \hat{P},\alpha\hat{P},\alpha^2\hat{P}]
\]
\end{theorem}

In order to construct $\hat{P}$ we first study the restriction of $\mathcal{D}_2$ to a subring $\mathcal{B}^{\prime}$ of $\mathcal{B}$, 
where $\mathcal{B}^{\prime}\cong k^{[6]}$.

\subsection{The Derivation $E$ in Dimension Six} Let $\mathcal{R}=k[x,y,s,t,u,v]=k^{[6]}$ and define the triangular derivation 
$E$ of $\mathcal{R}$ by:
\begin{equation}\label{tau}
v\to x^2y^2\,\, ,\,\, u\to y^3t\,\, ,\,\, t\to y^3s\,\, ,\,\, s\to x^3\,\, ,\,\, x\to 0\,\, ,\,\, y\to 0
\end{equation}
Then $E$ commutes with $\frac{\partial}{\partial v}$ and we let $\tau$ denote the restriction of $\frac{\partial}{\partial v}$ to $\krn E$. 
\begin{theorem}\label{tau-cable} There exists an infinite $\tau$-cable $\hat{\kappa}$ rooted at $x$.
\end{theorem}

\begin{proof}
According to equation (6) and Lemma 2 of \cite{Freudenburg.00}, there exists a sequence $w_n\in k[x,y,z,s,t,u]$, $n\ge 0$, with the following properties: Let $\pi_v:\mathcal{R}\to (\krn E)_{xy}$ be the Dixmier map for $E$ associated to the local slice $v$. 
\begin{itemize}
\item [(i)] $w_0=1$
\smallskip
\item [(ii)] $E^{3i}w_{3m}=(x^3y^3)^{2i}w_{3(m-i)}$ \quad $(m\ge 1, 0\le i\le m)$
\smallskip
\item [(iii)] $(-1)^{3m}\pi_v(xw_{3m})\in\mathcal{R}$ \quad $(m\ge 0)$
\end{itemize}
\medskip
Given $m\ge 0$, define $\kappa_{3m}\in\mathcal{R}$ by $\kappa_{3m}=(-1)^{3m}\pi_v(xw_{3m})$. 
By using (\ref{commute}) in {\it Sect.~2.1}, we see that for $m\ge 1$:
\begin{eqnarray*}
\frac{\partial^3}{\partial v^3}\kappa_{3m} &=& \frac{\partial^2}{\partial v^2} (-1)^{3m-1}\pi_v(xEw_{3m})\frac{\partial}{\partial v}\frac{v}{x^2y^2}\\
&=& \frac{\partial}{\partial v} (-1)^{3m-2}\pi_v(xE^2w_{3m})\frac{1}{x^2y^2}\frac{\partial}{\partial v}\frac{v}{x^2y^2}\\
&=&  (-1)^{3m-3}\pi_v(xE^3w_{3m})\frac{1}{x^4y^4}\frac{\partial}{\partial v}\frac{v}{x^2y^2}\\
&=& (-1)^{3m-3}\pi_v(x(x^3y^3)^2w_{3(m-1)})\frac{1}{x^6y^6}\\
&=& (-1)^{3(m-1)}\pi_v(xw_{3(m-1)})\\
&=& \kappa_{3(m-1)}
\end{eqnarray*}
Define:
\[
\kappa_{3m-1}=\frac{\partial}{\partial v}\kappa_{3m} \quad {\rm and}\quad \kappa_{3m-2}=\frac{\partial}{\partial v}\kappa_{3m-1} \quad (m\ge 1)
\]
Then $\hat{\kappa}:=(\kappa_n)$ is a $\tau$-cable rooted $x$. 
\end{proof}

\subsection{Proof of Theorem~\ref{Roberts}} 

Given $f_1,...,f_n\in\mathcal{B}$, recall that the Wronskian of $f_1,...,f_n$ relative to $\mathcal{D}_2$ is:
\[
W_{\mathcal{D}_2}(f_1,...,f_n) = \det \left( \mathcal{D}_2^if_j\right) \quad {\rm where} \quad 0\le i\le n-1\,\, ,\,\, 1\le j\le n
\]
See \cite{Freudenburg.06}, {\S}2.6. Define $F_1,F_2,F_3\in \mathcal{B}$ by:
\[
\textstyle F_1=S\,\, ,\,\, F_2=\frac{1}{2}W_{\mathcal{D}_2}(S,TU) \,\, ,\,\, F_3 = \frac{1}{6}X^{-3}W_{\mathcal{D}_2}(S,TU,STU)
\]
Then $\mathcal{D}_2$ restricts to the subring $\mathcal{B}^{\prime}=k[X,YZ,F_1,F_2,F_3,V]=k^{[6]}$, where:
\[
\mathcal{D}_2F_3 = (YZ)^3F_2\,\, ,\,\, \mathcal{D}_2F_2=(YZ)^3F_1\,\, ,\,\, \mathcal{D}_2F_1=X^3\,\, ,\,\, \mathcal{D}_2V=X^2(YZ)^2
\]
Therefore, setting $x=X$, $y=YZ$, $s=F_1$, $t=F_2$, $u=F_3$, and $v=V$, we see that the restriction of $\mathcal{D}_2$ to 
$\mathcal{B}^{\prime}$ equals $E$, as defined in line (\ref{tau}) above. 
By {\it Thm.~\ref{tau-cable}} there exists a $\delta_2$-cable $\hat{P}$ rooted at $X$ such that $\hat{P}\subset\mathcal{B}^{\prime}$. In particular, $\hat{P}=(P_i)$ has the form 
$P_i=\frac{1}{i!}XV^i + (\text{terms of lower degree in $V$})$. 

Consequently, $\alpha\hat{P}$ is a $\delta_2$-cable $\hat{P}$ rooted at $Y$, and $\alpha^2\hat{P}$ is a $\delta_2$-cable $\hat{P}$ rooted at $Z$. 
The proof is thus completed by applying Kuroda's result ({\it Thm.~\ref{Kuroda04b}} above). $\Box$

\begin{remark} {\rm It seems likely that the structure of $\mathcal{A}_2$ given in {\it Thm.~\ref{Roberts}} can be extended from $m=2$ to all $m\ge 2$. To do so by the method above requires a generalization of {\it Thm.~\ref{tau-cable}}. }
\end{remark}


\section{Further Comments and Questions}

\subsection{Tanimoto's Generators}\label{Tanimoto} In \cite{Tanimoto.06}, Tanimoto gives a set of generators for the ring $A$ by specifying a SAGBI basis consisting of $h$ together with homogeneous sequences $\lambda_n,\mu_n,\nu_n$ whose leading $v$-terms are $av^n,Fv^n$ and $Gv^n$, respectively. 
From {\it Cor.~\ref{dimension1}(a)} we see that $A$ is generated as a $k$-algebra by $h$ and the sequence $\lambda_n$, meaning $\mu_n$ and $\nu_n$ are redundant. Tanimoto also computed the Hilbert series for $A$, which is rational even though $A$ is not finitely generated. 

\subsection{Fundamental Problem for Cable Algebras}
\noindent If $B$ is an affine $k$-domain and $D\in {\rm LND}(B)$ is non-zero, then $B$ is a cable algebra and $(B,D)$ is a cable pair. We ask the following question, which we term the {\it Fundamental Problem for Cable Algebras}.
\begin{quote}
  Let $B$ be an affine $k$-domain and $D\in {\rm LND}(B)$. If $I_{\infty}\ne (0)$, does $B$ have an infinite $D$-cable? Equivalently, if every $D$-cable of $B$ is terminal, does $I_{\infty}=(0)$?
\end{quote}
Note that, if every $D$-cable of $B$ is terminal, then since $B$ is affine, there exist an integer $n\ge 1$ and terminal $D$-cables $\hat{t}_1,...,\hat{t}_n$ such that $B=k[\hat{t}_1,...,\hat{t}_n]$. 

\subsection{$Q$-Ideals} We would like know which $Q$-ideals are prime ideals of $\Omega$. For each even $n\ge 2$, consider the following statements regarding the fundamental $Q$-ideals.
\begin{itemize}
\item [(a)] $\mathcal{Q}_n$ is a prime ideal of $\Omega$
\smallskip
\item [(b)] ${\rm tr.deg}_k\Omega/\mathcal{Q}_n = \frac{n}{2}+1$
\smallskip
\item [(c)] $\Omega/\mathcal{Q}_n$ is a simple cable algebra over $k$
\end{itemize}
It is shown above that these are true statements for $n=2$ and $n=4$. Are these statements true for $n\ge 6$? 

\subsection{The Dimension Four Case}\label{dim4} In his famous paper \cite{Nagata.59}, Nagata presented the first counterexamples to Hilbert's Fourteenth Problem. In one of these, the transcendence degree of the ring of invariants over the ground field is four, and Nagata asked whether this could be reduced to three. In \cite{Kuroda.04a}, the second author gave an affirmative answer to Nagata's question in the form of the kernel of a derivation of $k^{[4]}$, but this derivation is not locally nilpotent (see also \cite{Kuroda.05b}). 

It remains an open question whether an algebraic $\G_a$-action on the polynomial ring $k^{[4]}$ always has a finitely generated ring of invariants. In \cite{Daigle.Freudenburg.01b} it is shown that this is the case for triangular actions, and this result was later generalized in 
\cite{Bhatwadekar.Daigle.09} to the case of actions having rank less than 4. The next natural case to consider is the case $T$ is a locally nilpotent derivation of $k^{[4]}$ of rank 4, and $T$ restricts to a coordinate subring $B=k^{[3]}$. If $k^{[4]}=B[v]$, then the partial derivative $\partial /\partial v$ restricts to $\krn T$. 
It is hoped that a good understanding of cable structures of invariant rings might lead to a complete solution of the dimension four case. 


\bigskip

\noindent \address{Department of Mathematics\\
Western Michigan University\\
Kalamazoo, Michigan 49008  USA}\\
\email{gene.freudenburg@wmich.edu}
\bigskip

\noindent\address{Department of Mathematics and Information Sciences\\
Tokyo Metropolitan University\\
Hachioji, Tokyo 192-0397 Japan} \\
\email{kuroda@tmu.ac.jp}

\end{document}